\documentclass[12pt, a4paper]{amsart}

\usepackage{amscd, amssymb, pb-diagram,enumerate,stmaryrd}
\usepackage{tikz}
\usepackage[margin=2.1cm]{geometry}
\usetikzlibrary{arrows,backgrounds,decorations.markings}

\def\clsp{\overline{\operatorname{span}}}
\def\range{\operatorname{range}}

\def\id{\operatorname{id}}
\def\MCE{\operatorname{MCE}}

\def\conv{\overline{\operatorname{conv}}}

\def\R{\mathbb{R}}
\def\N{\mathbb{N}}
\def\Z{\mathbb{Z}}
\def\T{\mathbb{T}}

\def\CC{\mathcal{C}}

\def\TT{\mathcal{T}}

\newcommand{\ep}{\varepsilon}

\newcommand{\floor}[1]{\lfloor #1\rfloor}
\newcommand{\ceil}[1]{\lceil #1\rceil}

\newtheorem{thm}{Theorem}[section]
\newtheorem{prop}[thm]{Proposition}
\newtheorem{lem}[thm]{Lemma}
\newtheorem{cor}[thm]{Corollary}
\theoremstyle{definition}
\newtheorem{dfn}[thm]{Definition}
\newtheorem{ntn}[thm]{Notation}
\theoremstyle{remark}
\newtheorem{rmk}[thm]{Remark}
\newtheorem{ex}[thm]{Example}

\author{Alex Kumjian}
\address{Alex Kumjian\\ Department of Mathematics (084)\\ University
of Nevada\\ Reno NV 89557-0084\\ USA} \email{alex@unr.edu}

\author{David Pask}
\address{David Pask, Aidan Sims, Michael F. Whittaker \\ School of Mathematics and
Applied Statistics  \\
The University of Wollongong\\
NSW  2522\\
AUSTRALIA} \email{dpask@uow.edu.au, asims@uow.edu.au, mfwhittaker@gmail.com}
\author{Aidan Sims}
\author{Michael F. Whittaker}

\thanks{This research was supported by the ARC}


\date{\today}

\title{Topological spaces associated to higher-rank graphs}

\begin{document}

\begin{abstract}
We investigate which topological spaces can be constructed as topological
realisations of higher-rank graphs. We describe equivalence relations on higher-rank
graphs for which the quotient is again a higher-rank graph, and show that identifying isomorphic
co-hereditary subgraphs in a disjoint union of two rank-$k$ graphs gives rise to pullbacks of
the associated $C^*$-algebras. We describe a combinatorial version of the
connected-sum operation and apply it to the rank-$2$-graph realisations
of the four basic surfaces to deduce that every compact 2-manifold is the topological
realisation of a rank-$2$ graph. We also show how to construct $k$-spheres and wedges of
$k$-spheres as topological realisations of rank-$k$ graphs.
\end{abstract}

\maketitle

\vspace{-5mm}

\section{Introduction}\label{sec:intro}

Higher-rank graphs, also called $k$-graphs, were introduced by Kumjian and Pask
\cite{KP2000} as combinatorial models for higher-rank Cuntz-Krieger algebras. Since then,
the resulting class of $C^*$-algebras has been studied in detail. More recently,
in \cite{PaskQuiggEtAl:NYJM04, PQR2, KKQS, KPS3}, an investigation of $k$-graphs
from a topological point of view was begun. Definition~3.2 of \cite{KKQS} associates to
each $k$-graph $\Lambda$ a topological realisation $X_\Lambda$ whose fundamental group
and homology are the same as the fundamental group and cubical homology of $\Lambda$. The
motivation for the current article was to investigate the range of topological spaces
which can be constructed as topological realisations of $k$-graphs. We began with two
goals: obtain all compact $2$-manifolds as topological realisations of $2$-graphs; and,
more generally, obtain all triangularisable $k$-manifolds as topological realisations of
$k$-graphs.

Our approach to the first goal was to exploit the classification of compact 2-manifolds
as spheres, $n$-holed tori, or connected-sums of the latter with the Klein bottle or
projective plane (see for example \cite[Theorem~I.7.2]{Massey}). Examples of $2$-graphs
whose topological realisations were homeomorphic to each of the four basic surfaces were
presented in \cite{KKQS}. So our aim was to develop a combinatorial connected-sum
operation for $2$-graphs, and to show that it can be applied to finite disjoint unions of
the four $2$-graphs just mentioned so as to construct any desired connected sum of their
topological realisations. We achieve this in Section~\ref{sec:surfaces}.

Our approach to the second, more general, goal consisted of four steps. Step~1 was to
determine which equivalence relations on $k$-graphs have the property that the quotient
category itself forms a $k$-graph. Step~2 was to invoke
\cite[Proposition~5.3]{KKQS}---which shows that topological realisation is a functor from
$k$-graphs to topological spaces---to see that, for $k$-graphs, topological realisations
of quotients coincide with quotients of topological realisations. Step~3 was to construct
$k$-graphs $\Sigma_k$ whose topological realisations are naturally homeomorphic to
$k$-simplices. Step~4 was to show that equivalence relations corresponding to desired
identifications amongst the $(k-1)$-faces of a disjoint union of copies of $\Sigma_k$ are
of the sort developed in Step~1; and then deduce that arbitrary triangularisable
manifolds could be realised as the topological realisations of appropriate quotients of
disjoint unions of copies of $\Sigma_k$. We have achieved steps 1--3, but the
combinatorics of $k$-graphs place significant constraints on the ways in which faces in a
disjoint union of copies of the $k$-graphs $\Sigma_k$ from Step~3 can be identified so as
to produce a new $k$-graph, so we are as yet unable to realise arbitrary triangularisable
manifolds. However, our construction is flexible enough so that we can glue two
$k$-simplices on their boundaries to obtain a $k$-sphere. The details of this appear in
section~\ref{sec:spheres}.

The paper is organised as follows. In Section~\ref{sec:quotients} we identify those equivalence relations $\sim$
on a $k$-graph $\Lambda$ for which the quotient $\Lambda/{\sim}$ forms a $k$-graph.
Proposition~\ref{prp:topological gluing} shows that the quotient operation is
well-behaved with respect to the topological realisation of a $k$-graph. In
Section~\ref{sec:Cstar} we investigate the properties of the quotients of $k$-graphs at
the level of their associated $C^*$-algebras. Specifically, given $k$-graphs $\Lambda_1$
and $\Lambda_2$ with a partially-defined isomorphism $\phi$ between the complements in
the $\Lambda_i$ of hereditary subsets $H_i$, the map $\phi$ induces an equivalence relation
$\sim_\phi$ on the disjoint union of $\Lambda_1$ and $\Lambda_2$ and we show that the quotient
$(\Lambda_1 \sqcup \Lambda_2)/{\sim_\phi}$ forms a $k$-graph. We then show that the Toeplitz
algebra $\TT C^*(\Lambda_1 \sqcup \Lambda_2)/{\sim_\phi})$ is a pullback of
$\TT C^*(\Lambda_1)$ and $\TT C^*(\Lambda_2)$ over the $C^*$-algebra of the common subgraph (see
Theorem~\ref{thm:pullback}). If the $H_i$ are also saturated, then this result also descends to
Cuntz-Kreiger algebras (Corollary~\ref{cor:CK pullback}).

In section~\ref{sec:surfaces} we define the connected-sum operation on $2$-graphs and
show that it corresponds to the connected-sum operation on their topological
realisations. Then, after recalling from \cite{KKQS} how to realise the four basic
surfaces, we show that every compact surface is the topological realisation of a $2$-graph.

In section~\ref{sec:spheres} we construct, for each $k \in \N$, a $k$-graph $\Sigma_k$
whose topological realisation is a $k$-dimensional simplex (see
Theorem~\ref{thm:homeomorphism}). The combinatorics involved in this construction are
interesting in their own right---the vertices of $\Sigma_k$ are indexed by \emph{placing
functions} which can be thought of as the possible outcomes of a horse race involving
$k+1$ horses (see the On-Line Encyclopedia of Integer Sequences \cite[Sequence
A000670]{oeis}); alternatively, they can be regarded as ordered partitions of a
$(k+1)$-element set. Using suitable equivalence relations, two (or more) copies of
$\Sigma_k$ can be glued along their common boundary to produce a new $k$-graph. In
Theorem~\ref{thm:spheres realised} we use this construction to realise all $k$-spheres as
the topological realisations of $k$-graphs; we also show that for each $n \ge 1$ there is
a finite $k$-graph $\Lambda$ whose topological realisation is a wedge of $n$ $k$-spheres.

\subsection*{Background and notation}
We regard $\N^k$ as a semigroup under addition, with identity 0 and generators $e_1,
\dots, e_k$. For $m,n \in \N^k$, we write $m_i$ for the $i$\textsuperscript{th}
coordinate of $m$, and we define $m \vee n \in \N^k$ by $(m \vee n)_i = \max\{m_i,
n_i\}$. We write $m \le n$ if and only if $m_i \le n_i$ for all $i$.

Let $\Lambda$ be a countable small category and $d : \Lambda \to \N^k$ a functor.
Write $\Lambda^n := d^{-1}(n)$ for each $n \in \N^k$. Then $\Lambda$ is a
$k$-graph if $d$ satisfies the \emph{factorisation property}: $(\mu,\nu) \mapsto \mu\nu$
is a bijection of $\{(\mu,\nu) \in \Lambda^m \times \Lambda^n : s(\mu) = r(\nu)\}$ onto
$\Lambda^{m+n}$ for each $m,n \in \N^k$ (see~\cite{KP2000}). We then have $\Lambda^0 = \{\id_o : o\in
\operatorname{Obj}(\Lambda)\}$, and so we regard the domain and codomain maps as maps $s,
r : \Lambda \to \Lambda^0$.

If $\Lambda_i$ is a $k_i$-graph for $i=1,2$ then the \emph{cartesian product} $\Lambda_1
\times \Lambda_2$ with the natural product structure forms a $(k_1 + k_2)$-graph (see
\cite[Proposition 1.8]{KP2000}).

A \emph{$k$-graph morphism} between $k$-graphs $(\Lambda, d_1)$ and $(\Sigma, d_2)$ is a
functor $\phi : \Lambda \to \Sigma$ such that $d_1 ( \lambda ) = d_2 ( \phi ( \lambda ))$
for all $\lambda \in \Lambda$.

Recall from \cite{PaskQuiggEtAl:NYJM04} that for $v,w \in \Lambda^0$ and $X \subseteq
\Lambda$, we write
\[
v X := \{\lambda \in X : r(\lambda) = v\},\quad
X w := \{\lambda \in X : s(\lambda) = w\},\quad\text{and}\quad
v X w = vX \cap Xw.
\]
If $V \subset \Lambda^0$, then $V\Lambda = r^{-1} (V)$ and $\Lambda V = s^{-1} (V)$. A
$k$-graph $\Lambda$ has \emph{no sources} if $0 < |v\Lambda^n|$ for all $v \in \Lambda^0$
and $n \in \N^k$.

For $\lambda \in \Lambda$ and $0 \le m \le n \le d(\lambda)$, the factorisation property
yields unique elements $\alpha \in \Lambda^m$, $\beta \in \Lambda^{n-m}$ and $\gamma \in
\Lambda^{d(\lambda) - n}$ such that $\lambda = \alpha\beta\gamma$. Define $\lambda(m,n)
:= \beta$. We then have $\lambda(0,m) = \alpha$ and $\lambda(n, d(\lambda)) = \gamma$. In
particular, $\lambda = \lambda(0,m)\lambda(m, d(\lambda))$ for each $0 \le m \le d(\lambda)$.

Given $\mu,\nu \in \Lambda$, we define $\MCE( \mu,\nu):=\{\lambda \in \Lambda^{d(\mu)
\vee d(\nu)} : \lambda = \mu\mu' = \nu\nu'\text{ for some }\mu',\nu'\}$.  If $| \MCE (
\mu , \nu ) | < \infty$ for all $\mu , \nu \in \Lambda$ then we say that $\Lambda$ is
\emph{finitely aligned}.

We define $\Lambda^{*2} = \{(\lambda, \mu) \in \Lambda \times \Lambda : s(\lambda) =
r(\mu)\}$, the collection of composable pairs in $\Lambda$. If $\Lambda_1 , \Lambda_2$
are $k$-graphs, then the disjoint union $\Lambda_1 \sqcup \Lambda_2$ naturally forms a
$k$-graph. We will allow for the possibility of $0$-graphs with the convention that
$\N^0$ is the trivial semigroup $\{0\}$. We insist that all $k$-graphs are nonempty. See
\cite{KP2000} for further details regarding the basic structure of $k$-graphs.

A set of vertices $H \subset \Lambda^0$ is \emph{hereditary} if $s(H\Lambda) \subseteq
H$; similarly, $H$ is \emph{co-hereditary} if $r(\Lambda H) \subseteq H$.  A set $H$ is
hereditary if and only if $\Lambda^0 \setminus H$ is co-hereditary.  A subgraph
$\Gamma\subset\Lambda$ is said to be \emph{hereditary} (resp.~\emph{co-hereditary}) if
$\Gamma = \Gamma^0\Lambda$ (resp.\ $\Gamma = \Lambda\Gamma^0$).  If $\Gamma$ is a
hereditary subgraph, then $\Gamma^0$ is a hereditary subset of $\Lambda^0$. Applying this
result to the opposite category $\Lambda^{\mathrm{op}}$, which is also a $k$-graph under
the same degree map, yields the corresponding statement for a co-hereditary subgraph.

We now review the construction of the topological realisation of a $k$-graph for $k \ge
1$ given in \cite{KKQS}. Given $t \in \R^k$, we will write $\ceil{t}$ for the least
element of $\Z^k$ which is coordinatewise greater than  or equal to $t$, and $\floor{t}$
for the greatest element of $\Z^k$ which is coordinatewise less than or equal to $t$. Let
$\mathbf{1}_k := (1,1, \dots, 1) \in \N^k$. Then $\floor{t} \le t \le \ceil{t} \le
\floor{t} + \mathbf{1}_{k}$ for all $t \in \R^k$.

Given $p \le q \in \N^k$, we denote by $[p,q]$ the \emph{closed interval} $\{t \in \R^k :
p \le t \le q\}$, and we denote by $(p,q)$ the \emph{relatively open interval} $\{t \in
[p,q] : p_i < t_i < q_i\text{ whenever } p_i < q_i\}$. Observe that $(p,q)$ is not open
in $\R^k$ unless $p_i < q_i$ for all $i$, but it is open as a subspace of $[p,q]$. The
set $(p,q)$ is never empty: for example if $p = q$ then $(p,q) = [p,q] = \{p\}$. In
general, as a subset of $\R^k$, the dimension of $(p,q)$ is $|\{i \le k : p_i < q_i\}|$.
If $p_i < q_i$ then the $i$\textsuperscript{th}-coordinate projection of $(p,q)$ is
$(p_i, q_i)$, and if $p_i = q_i$ then the $i$\textsuperscript{th}-coordinate projection
of $(p,q)$ is $\{p_i\}$. If $m \in \N^k$ with $m \le \mathbf{1}_{k}$, then $\floor{t}=0$
and $\ceil{t} = m$ for all $t \in (0, m)$.

We define a relation on the topological disjoint union $\bigsqcup_{\lambda \in \Lambda}
\{\lambda\} \times [0,d(\lambda)]$ by
\begin{equation}\label{eq:equiv rel}
(\mu, s) \sim (\nu, t) \quad\iff\quad \mu(\floor{s}, \ceil{s})
    = \nu(\floor{t}, \ceil{t})\text{ and } s - \floor{s} = t - \floor{t}.
\end{equation}
It is straightforward to see that this is an equivalence relation.

\begin{dfn}[cf.\,{\cite[Definition 3.2]{KKQS}}]
Let $\Lambda$ be a $k$-graph. With notation as above, we define the \emph{topological
realisation} $X_\Lambda$ of $\Lambda$ to be the quotient space
\[
\Big(\bigsqcup_{\lambda \in \Lambda} \{\lambda\} \times [0, d(\lambda)]\Big)\Big/\sim .
\]
\end{dfn}

We will need to discuss topological realisations of $0$-graphs in Section~\ref{sec:spheres}.
To avoid discussing separate cases, we take the convention that
$\mathbf{1}_0 = 0 \in \{0\} = \N^0$. If $\Lambda$ is a $0$-graph, then $\Lambda = \Lambda^0$,
the relation $\sim$ of equation~\eqref{eq:equiv rel} is trivial, and we have
$X_\Lambda = \Lambda^0 \times \{0\} \cong \Lambda^0$.

\begin{rmk}
One can view the morphisms in a $k$-graph whose degrees are smaller than $\mathbf{1}_{k}$
as a cubical set in a natural way \cite[Theorem~A.4]{KPS3}, and then the topological
realisation $X_\Lambda$ of $\Lambda$ is the same as the topological realisation of the
associated cubical set (see, for example, \cite{Grandis2005}).
\end{rmk}

\section{Quotients of $k$-graphs}\label{sec:quotients}

In this section we identify the equivalence relations $\sim$ on $k$-graphs $\Lambda$ for
which the quotient set $\Lambda/\sim$ itself becomes a $k$-graph. We describe the
topological realisation of $\Lambda/\sim$ as a quotient of the topological realisation of
$\Lambda$.

\begin{prop}\label{prp:quotient k-graph}
Let $\Lambda$ be a $k$-graph. Suppose that $\sim$ is an equivalence relation on $\Lambda$
with the following properties:
\begin{enumerate}
\item\label{it:sim-d} if $\mu \sim \nu$ then $d(\mu) = d(\nu)$;
\item\label{it:sim-comp} if $\alpha \sim \alpha'$ and $\beta \sim \beta'$ with
    $r(\beta) = s(\alpha)$ and $r(\beta') = s(\alpha')$, then $\alpha\beta \sim
    \alpha'\beta'$;
\item\label{it:sim-factor} if $\alpha\beta \sim \alpha'\beta'$ and $d(\alpha) =
    d(\alpha')$, then $\alpha \sim \alpha'$ and $\beta \sim \beta'$;
\item\label{it:sim-lift} if $s(\alpha) \sim r(\beta)$, then there exist $\alpha',
    \beta'$ such that $\alpha' \sim \alpha$, $\beta'\sim\beta$ and $s(\alpha') =
    r(\beta')$.
\end{enumerate}
Then the structure maps on $\Lambda$ descend to structure maps on $\Lambda/{\sim}$ under
which the latter is a $k$-graph.
\end{prop}

\begin{proof}
Let $[\lambda]$ denote the equivalence class containing $\lambda$ and suppose that $\mu
\in [\lambda]$. Note that $d(\lambda) = d(\mu)$ by~\eqref{it:sim-d}.
We have $\lambda = r(\lambda)\lambda$ and $\mu = r(\mu)\mu$, and since $d(r(\lambda)) = 0 =
d(r(\mu))$, condition~(\ref{it:sim-factor}) implies that $r(\lambda) \sim r(\mu)$. Since
$\lambda = \lambda s(\lambda)$, $\mu = \mu s(\mu)$ and $d(\lambda) = d(\mu)$, condition
(\ref{it:sim-factor}) implies that $s(\lambda) \sim s(\mu)$.
Hence, the formulas
\[
d([\mu]) := d(\mu), \quad r([\mu]) = [r(\mu)], \quad\text{ and } s([\mu]) = [s(\mu)]
\]
are well defined.

If $[s(\lambda)] = [r(\mu)]$, then condition~(\ref{it:sim-lift}) implies that there
exist $\alpha \in [\lambda]$ and $\beta \in [\mu]$ such that $s(\alpha) = r(\beta)$, and then
condition~(\ref{it:sim-comp}) implies that $[\alpha\beta]$ does not depend on the choice
of $\alpha$ and $\beta$; so we may define $[\lambda][\mu] := [\alpha\beta]$. To see that
this is associative, suppose that $[\lambda], [\mu], [\nu]$ is a composable triple in
$\Lambda/{\sim}$. By~(\ref{it:sim-lift}) there exist $\alpha \in [\lambda]$ and $\beta
\in [\mu]$ with $s(\alpha) = r(\beta)$. Now $s([\alpha\beta]) = r([\nu])$, and
so~(\ref{it:sim-lift}) again gives $\eta \in [\alpha\beta]$ and $\gamma \in [\nu]$ such
that $s(\eta) = r(\nu)$. The factorisation property gives $\eta = \alpha'\beta'$ with
$d(\alpha') = d(\alpha)$ and $d(\beta') = d(\beta)$, and then~(\ref{it:sim-factor})
implies that $\alpha' \sim \alpha$ and $\beta' \sim \beta$. We now have
\[
([\lambda][\mu])[\nu]
    = ([\alpha'][\beta'])[\gamma]
    = [\alpha'\beta'][\gamma]
    = [\alpha'\beta'\gamma]
    = [\alpha'][\beta'\gamma]
    =[\lambda]([\mu][\nu]).
\]

We have now established that $\Lambda/{\sim}$ is a category. For the unique factorisation
property, suppose that $d([\lambda]) = m+n$. Then $d(\lambda) = m+n$, and the
factorisation property in $\Lambda$ allows us to write $\lambda = \mu\nu$ with $d(\mu) =
m$ and $d(\nu) = n$. We then have $[\lambda] = [\mu][\nu]$ with $d([\mu]) = m$ and
$d([\nu]) = n$, and condition~(\ref{it:sim-factor}) implies that this factorisation is
unique.
\end{proof}

\begin{ex} \label{ex:disjoint_union_equiv}
Let $\Lambda_1 , \Lambda_2$ be $k$-graphs and $\phi_i : \Gamma \to \Lambda_i$ be an
injective $k$-graph homomorphism for $i=1,2$ such that $\phi_i ( \Gamma )$ is hereditary
(or cohereditary) in $\Lambda_i$ for $i=1,2$. Routine checks show that the smallest
equivalence relation $\sim_\phi$ on $\Lambda_1 \sqcup \Lambda_2$ such that
$\phi_1(\gamma) \sim_\phi \phi_2(\gamma)$ for all $\gamma \in \Gamma$ satisfies the
hypotheses of Proposition~\ref{prp:quotient k-graph}.
\end{ex}

\begin{prop}\label{prp:topological gluing}
Under the hypotheses of Proposition~\ref{prp:quotient k-graph}, there is an equivalence
relation $\approx$ on $X_\Lambda$ such that $[\mu,s] \approx [\nu,t]$ if and only if
$\mu(\floor{s}, \ceil{s}) \sim \nu(\floor{t}, \ceil{t})$ and $s - \floor{s} = t -
\floor{t}$. Let $\llbracket\lambda, t\rrbracket$ denote the equivalence class of
$[\lambda,t] \in X_\Lambda$ under $\approx$, and let $[\lambda]_\sim$ denote the
equivalence class of $\lambda \in \Lambda$ under $\sim$. Then there is a homeomorphism
$X_\Lambda /{\approx} \cong X_{\Lambda/{\sim}}$ satisfying $\llbracket\lambda, t
\rrbracket \mapsto [[\lambda]_\sim, t]$ for all $\lambda, t$.
\end{prop}

\begin{proof}
Recall that in $X_\Lambda$, we have $[\mu, s] = [\mu',s']$ if and only if
$\mu(\floor{s},\ceil{s}) = \mu'(\floor{s'}, \ceil{s'})$ and $s - \floor{s} = s' -
\floor{s'}$. So the formula for $\approx$ is well-defined and determines a relation on
$X_\Lambda$. It is elementary to check that this is an equivalence relation.

The map $\lambda \mapsto [\lambda]_{\sim}$ is a surjective $k$-graph morphism from
$\Lambda$ to $\Lambda/{\sim}$, and so \cite[Proposition~5.3]{KKQS} implies that there is
a continuous surjection $\phi : X_\Lambda \to X_{\Lambda/{\sim}}$ satisfying $[\lambda, t]
\mapsto \big[[\lambda]_{\sim}, t\big]$. We need to show that $\big[[\lambda]_{\sim}, t\big]
= \big[[\lambda']_{\sim}, t'\big]$ if and only if  $[\lambda, t] \approx [\lambda', t']$.
Suppose that $\big[[\lambda]_{\sim}, t \big] = \big[[\lambda']_{\sim}, t'\big]$. Then
$[\lambda]_{\sim} (\floor{t}, \ceil{t}) = [\lambda']_{\sim}
(\floor{t'}, \ceil{t'})$ and $t - \floor{t} = t' - \floor{t'}$. Hence,
\[
[\lambda(\floor{t}, \ceil{t})] = [\lambda]_{\sim} (\floor{t}, \ceil{t}) = [\lambda']_{\sim} (\floor{t'}, \ceil{t'})
 = [\lambda'(\floor{t'}, \ceil{t'})].
\]
Thus $\lambda(\floor{t}, \ceil{t}) \sim \lambda'(\floor{t'}, \ceil{t'})$, and so
$[\lambda, t] \approx [\lambda', t']$. Reversing the steps above proves the converse.
\end{proof}

\section{The Toeplitz algebras of quotients of $k$-graphs}\label{sec:Cstar}

In this section we study the $C^*$-algebras associated to quotients of
$k$-graphs as discussed in the preceding section. There are two $C^*$-algebras associated
to a finitely-aligned $k$-graph $\Lambda$ with no sources: the Toeplitz algebra $\TT C^*
( \Lambda )$ and the ``usual" $k$-graph algebra $C^* ( \Lambda )$.

Recall from \cite{RaeburnSims:JOT05} that the Toeplitz algebra $\TT C^*(\Lambda)$ of a
finitely-aligned $k$-graph $\Lambda$ is the universal $C^*$-algebra generated by elements
$\{t_\lambda : \lambda \in \Lambda\}$ such that
\begin{itemize}
\item[(TCK1)] $\{t_v : v \in \Lambda^0\}$ is a set of mutually orthogonal
    projections;
\item[(TCK2)] $t_\mu t_\nu = t_{\mu\nu}$ whenever $s(\mu) = r(\nu)$;
\item[(TCK3)] $t^*_\mu t_\mu = t_{s(\mu)}$ for all $\mu$;
\item[(TCK4)] for every $v \in \Lambda^0$, $n \in \N^k$ and finite $F \subseteq
    v\Lambda^n$, we have $t_v \ge \sum_{\mu \in F} t_\mu t^*_\mu$; and
\item[(TCK5)] $t_\mu t^*_\mu t_\nu t^*_\nu = \sum_{\lambda \in \MCE(\mu,\nu)}
    t_\lambda t^*_\lambda$ for all $\mu,\nu$ (an empty sum is interpreted as zero).
\end{itemize}
A collection $t = \{t_\lambda : \lambda \in \Lambda\}$ satisfying (TCK1)--(TCK5) is
called a \emph{Toeplitz-Cuntz-Krieger $\Lambda$-family}.

To describe the ``usual" $k$-graph algebra $C^*(\Lambda)$, first recall
that for $v \in \Lambda^0$, a nonempty set $E \subseteq v\Lambda$ is \emph{exhaustive} if for
every $\mu \in v\Lambda$ there exists $\lambda \in E$ such that $\MCE(\mu,\lambda) \not=
\emptyset$; equivalently, $E$ is exhaustive if $\mu\Lambda \cap E\Lambda \not= \emptyset$
for all $\mu \in v\Lambda$. The $C^*$-algebra
$C^*(\Lambda)$ is universal for Toeplitz-Cuntz-Krieger $\Lambda$-families satisfying the
additional relation:
\begin{itemize}
\item[(CK)] for every $v \in \Lambda^0$ and finite exhaustive $E \subseteq v\Lambda$,
    we have $\prod_{\lambda \in E} ( s_v - s_\lambda s^*_\lambda ) = 0$
\end{itemize}
(see \cite{RaeburnSimsEtAl:JFA04}).

We now describe how inclusions of $k$-graphs induce homomorphisms
of their Toeplitz algebras. The following is a generalisation of \cite[Proposition
3.3]{PRY}, \cite[Theorem 5.2]{A}\footnote{there is a missing injectivity hypothesis in
the statement of this result} and \cite[Lemma 2.3]{MPR2013}.

\begin{lem}\label{lem:k-graph inclusion}
Let $\Lambda$ be a finitely-aligned $k$-graph with no sources and suppose that $\Gamma
\subseteq \Lambda$ is a subgraph such that $\MCE_\Lambda(\mu,\nu) \subseteq \Gamma$ for
all $\mu,\nu \in \Gamma$. Let $\{ s_\gamma : \gamma \in \Gamma \}$ denote the universal
Toeplitz-Cuntz-Krieger $\Gamma$-family and let $\{ t_\lambda : \lambda \in \Lambda \}$
denote the universal Toeplitz-Cuntz-Krieger $\Lambda$-family. There is an injective
homomorphism $\iota : \TT C^*(\Gamma) \to \TT C^*(\Lambda)$ such that $\iota(s_\gamma ) =
t_\gamma$ for all $\gamma \in \Gamma$.
\end{lem}
\begin{proof}
The elements $\{t_\gamma : \gamma \in \Gamma\}$ form a Toeplitz-Cuntz-Krieger
$\Gamma$-family in $\TT C^* ( \Lambda )$: the relations all follow from the same
relations for $\Lambda$ (the hypothesis that $\MCE_\Lambda(\mu,\nu) \subseteq \Gamma$ for
all $\mu,\nu \in \Gamma$ ensures that condition (TCK5) holds). So the universal property
of $\TT C^*(\Gamma)$ induces a homomorphism $\iota : \TT C^*(\Gamma) \to \TT
C^*(\Lambda)$ satisfying $\iota(s_\gamma ) = t_\gamma$. Theorem~3.11 of \cite{SWW}
applied to $\Lambda$ implies that $\prod_{\lambda \in E} (t_v - t_\lambda t^*_\lambda)
\not= 0$ for every $v \in \Lambda^0$ and every finite $E \subseteq v\Lambda$. The reverse
implication of the same theorem applied to $\Gamma$ then implies that $\iota$ is
injective.
\end{proof}

\begin{lem}\label{lem:exact sequence}
Suppose that $\Lambda$ is a finitely aligned $k$-graph with no sources and that $H
\subseteq \Lambda^0$ is hereditary. Let $T = \Lambda^0 \setminus H$. Then $H\Lambda$ is a
subgraph such that $\MCE_\Lambda(\mu,\nu) \subseteq H \Lambda$ for all
$\mu,\nu \in H \Lambda$, and $\Lambda T$ is subgraph of
$\Lambda$. Let $\iota : \TT C^*(H\Lambda) \to \TT C^*(\Lambda)$ be the homomorphism of
Lemma~\ref{lem:k-graph inclusion}. There is a homomorphism $\pi : \TT C^*(\Lambda) \to
\TT C^*(\Lambda T)$ satisfying
\[
\pi(t_\lambda)
    = \begin{cases}
        s_\lambda &\text{ if $s(\lambda) \in T$}\\
        0 &\text{ otherwise.}
    \end{cases}
\]
The sum $\sum_{v \in H} t_v$ converges strictly to a full multiplier projection $P_H$ of
$\ker(\pi)$. We have
\[
\ker(\pi) = \clsp\{s_\mu s^*_\nu : s(\mu) = s(\nu) \in H\}\quad\text{ and }\quad
P_H \ker(\pi) P_H = \iota(\TT C^*(H\Lambda)).
\]
\end{lem}
\begin{proof}
Establishing the properties of $H \Lambda$ and $\Lambda T$ is straightforward. Let $I_H$
be the ideal of $\TT C^*(\Lambda)$ generated by $\{p_v : v \in H\}$. Theorem~4.4 of
\cite{SWW} applied with $\mathcal{E} = \emptyset$ and $c \equiv 1$ shows that there is an isomorphism $\TT C^*(\Lambda)/I_H \cong \TT C^*(\Lambda T)$ which carries $t_\lambda + I_H$ to $s_\lambda$
for $\lambda \in \Lambda T$. Composing this with the quotient map from $\TT C^*(\Lambda)$
to $\TT C^*(\Lambda)/I_H$ gives the desired homomorphism $\pi$. It is routine to check that
the sum $\sum_{v \in H} t_v$ converges to a multiplier of $\TT C^*(\Lambda)$ (see
\cite[Lemma 2.1]{MPR2013} or \cite[Lemma 1.2]{BatesPaskEtAl:NYJM00}). So $\ker(\pi)
= I_H$, and $P_H \ker(\pi) P_H$ is a full hereditary subalgebra
of $\ker(\pi)$. Since $H$ is hereditary, $\clsp\{s_\mu s^*_\nu : s(\mu) = s(\nu) \in H\}$ is an ideal which is clearly contained in $I_H$ and contains all its generators, so the two are equal.
We have $P_H s_\lambda = s_\lambda$ if $r(\lambda) \in H$ and $P_H s_\lambda = 0$
otherwise, and so
\[
P_H \ker(\pi) P_H = P_H \TT C^*(\Lambda) P_H = \clsp\{s_\mu s^*_\nu : \mu,\nu \in H\Lambda\} =
\iota(\TT C^*(H\Lambda)).\qedhere
\]
\end{proof}

Recall from \cite[Section~2.2]{Pedersen} (see also \cite[\S15.3]{Blackadar:K-theory}) that if $B_1,
B_2$ are $C^*$-algebras and $q_i : B_i \to C$ is a homomorphism for each $i$, then the
\emph{pullback} $B_1 \oplus_C B_2$ is the subalgebra $\{(a,b) \in B_1 \oplus B_2 : q_1(a)
= q_2(b)\}$ of $B_1 \oplus B_2$. It has the universal property that
\begin{enumerate}
\item  the canonical maps $\pi_i : B_1 \oplus_C B_2 \to B_i$ satisfy $q_1 \circ \pi_1 = q_2 \circ \pi_2$ and $\ker \pi_1 \cap \ker \pi_2 = \{0\}$; and
\item if $\psi_i : A \to B_i$ are homomorphisms such that $q_1 \circ \psi_1 = q_2 \circ \psi_2$, then there is a unique homomorphism $\psi : A \to B_1 \oplus_C B_2$ for which the following diagram commutes:
\end{enumerate}
\[
\begin{tikzpicture}[>=stealth]
    \node (C) at (3,0) {$C$};
    \node (B_1) at (0,0) {$B_1$};
    \node (B_2) at (3,3) {$B_2$};
    \node (sum) at (1.5,1.5) {$B_1 \oplus_C B_2$};
    \node (A) at (0,3) {$A$};
    \draw[->] (B_1)--(C) node[anchor=south, pos=0.5, inner sep=0.5pt] {\small$q_1$};
    \draw[->] (B_2)--(C) node[anchor=west, pos=0.5, inner sep=0.5pt] {\small $q_2$};
    \draw[->] (A)--(B_1) node[anchor=east, pos=0.5, inner sep=0.5pt] {\small$\psi_1$};
    \draw[->] (A)--(B_2) node[above, pos=0.5, inner sep=0.5pt] {\small$\psi_2$};
    \draw[->] (sum)--(B_1) node[anchor=north west, pos=0.5, inner sep=0.5pt] {\small$\pi_1$};
    \draw[->] (sum)--(B_2) node[anchor=north west, pos=0.5, inner sep=0.5pt] {\small$\pi_2$};
    \draw[dashed,->] (A)--(sum) node[anchor=north east, pos=0.5, inner sep=0.5pt] {\small$\psi$};
\end{tikzpicture}
\]

\begin{thm}\label{thm:pullback}
Let $\Lambda_1, \Lambda_2$ and $\Gamma$ be finitely aligned $k$-graphs with no sources.
Suppose that for $i = 1,2$ we have an injective $k$-graph morphism $\phi_i : \Gamma
\hookrightarrow \Lambda_i$ such that $\phi_i (\Gamma^0)$ is a co-hereditary subgraph of
$\Lambda_i$. Let $\pi_i : \TT C^*(\Lambda_i) \to \TT C^*(\Gamma)$ be the homomorphism
obtained from Lemma~\ref{lem:exact sequence} and the isomorphism of $~\Gamma$ with
$\phi_i(\Gamma)$, and form the pullback $C^*$-algebra $\TT C^*(\Lambda_1) \oplus_{\TT
C^*(\Gamma)} \TT C^*(\Lambda_2)$ with respect to $\pi_1,\pi_2$. Let $\phi :
\phi_1(\Gamma) \to \phi_2(\Gamma)$ be the isomorphism $\phi_1(\lambda) \mapsto
\phi_2(\lambda)$, and let $\sim_\phi$ be the equivalence relation of
Example~\ref{ex:disjoint_union_equiv}. For $i = 1,2$, let $\{s^i_\lambda : \lambda \in \Lambda_i\}$
denote the universal Toeplitz-Cuntz-Krieger family in $\TT C^*(\Lambda_i)$, and let
$\{s_{[\lambda]} : \lambda \in (\Lambda_1 \sqcup \Lambda_2)/{\sim_\phi}\}$ be the universal
generating family in $\TT C^*((\Lambda_1 \sqcup \Lambda_2)/{\sim_\phi})$. Then there is an
isomorphism
\[
\theta : \TT C^*((\Lambda_1 \sqcup \Lambda_2)/{\sim_\phi}) \to
\TT C^*(\Lambda_1) \oplus_{\TT C^*(\Gamma)} \TT C^*(\Lambda_2)
\]
such that
\begin{equation} \label{eq:thetadef}
\theta(s_{[\lambda]}) = \begin{cases}
    (s^1_\lambda, 0) &\text{ if $\lambda \in \Lambda_1 \setminus \Gamma$}\\
    (0, s^2_\lambda) &\text{ if $\lambda \in \Lambda_2 \setminus \Gamma$}\\
    (s^1_{\phi_1(\gamma)}, s^2_{\phi_2(\gamma)}) &\text{ if $[\lambda] = \{\phi_1(\gamma), \phi_2(\gamma)\}$.}
\end{cases}
\end{equation}
\end{thm}
\begin{proof}
Let $\Sigma := (\Lambda_1 \sqcup \Lambda_2)/{\sim_\phi}$. The sets $H_1 := \{[v] : v \in
\Lambda_1^0 \setminus \phi_1(\Gamma)\}$ and $H_2 := \{[v] : v \in \Lambda_2^0 \setminus
\phi_2(\Gamma)\}$ are hereditary in $\Lambda_1$ and $\Lambda_2$ respectively. Let $T_i :=
\Sigma^0 \setminus H_i$ for $i=1,2$. Then $\lambda \mapsto [\lambda]$ is an isomorphism
of $\Lambda_i$ onto $\Sigma T_{3-i}$ for $i = 1,2$, and so Lemma~\ref{lem:exact sequence}
implies that there are homomorphisms $\psi_i : \TT C^*(\Sigma) \to \TT C^*(\Lambda_i)$
such that $\psi_i(t_{[\lambda]}) = s_\lambda$ for $\lambda \in \Lambda_i$. We then have
\begin{equation}\label{eq:theta formula}
\pi_1 \circ \psi_1(t_{[\lambda]})
    = \begin{cases}
        s_\lambda &\text{ if $\lambda \in \Gamma$}\\
        0 &\text{ otherwise,}
    \end{cases}
\end{equation}
The universal property of the pullback now implies that there exists a homomorphism $\theta : \TT C^*(\Sigma) \to \TT C^*(\Lambda_1) \oplus_{\TT C^*(\Gamma)} \TT C^*(\Lambda_2)$ satisfying \eqref{eq:thetadef}.

To see that $\theta$ is an isomorphism, we will invoke Proposition~3.1 of \cite{Pedersen} to see that $\TT C^*(\Sigma)$ is itself a pullback. We must show that
\begin{enumerate}
\item\label{it:ker int 0} $\ker(\psi_1) \cap \ker(\psi_2) = 0$,
\item\label{it:compatible} $\pi_2^{-1}(\pi_1(\TT C^*(\Lambda_1)) = \psi_2(\TT C^*(\Sigma))$, and
\item\label{it:exact} $\psi_1(\ker(\psi_2)) = \ker(\pi_1)$.
\end{enumerate}
For~(\ref{it:ker int 0}), observe first that since $\psi_1$ and $\psi_2$ are equivariant
for the gauge actions on $\TT C^*(\Sigma)$ and the $\TT C^*(\Lambda_i)$, the ideal
$\ker\psi_1 \cap \ker\psi_2$ is gauge invariant. Each $[v] \in \Sigma^0$ belongs to
either $\Lambda_1^0$ or $\Lambda_2^0$, and so no $p_{[v]}$ belongs to $\ker\psi_1 \cap
\ker\psi_2$. If $[v] \in \Sigma^0$ and $F \subseteq [v]\Sigma \setminus \{[v]\}$ is
finite, then $v \in \Lambda_i^0$ for some $i$. Now
\[
\psi_i\Big(\prod_{[\mu] \in F} (p_{[v]} - s_{[\mu]} s^*_{[\mu]})\Big)
    = \prod_{\mu \in \Lambda_1, [\mu] \in F} (p^i_{v} - s^i_{\mu} (s^i_{\mu})^*) \not= 0.
\]
So $\prod_{[\mu] \in F} (p_{[v]} - s_{[\mu]} s^*_{[\mu]}) \not\in \ker\psi_1 \cap \ker\psi_2$. Theorem~4.6 of \cite{SWW} implies that $\ker\psi_1 \cap \ker\psi_2 = \{0\}$.

For~(\ref{it:compatible}), the containment $\psi_2(\TT C^*(\Sigma)) \subseteq \pi_2^{-1}(\pi_1(\TT C^*(\Lambda_1)))$ is immediate because $\pi_1 \circ \psi_1 = \pi_2 \circ \psi_2$, and the reverse containment is clear because $\psi_2$ is surjective.

For~(\ref{it:exact}), observe that Lemma~\ref{lem:exact sequence} implies that
\[
\ker(\psi_2) = \clsp\big\{s_{[\mu]} s^*_{[\nu]} : s([\mu]) = s([\nu]) \in \{[v] : v \in \Lambda_1^0 \setminus \phi_1(\Gamma^0)\}\big\},
\]
and that $\ker(\pi_1) = \clsp\{s^1_\mu (s^1_\nu)^* : s(\mu) = s(\nu) \in \Lambda_1^0
\setminus \phi_1(\Gamma^0)\}$. If $s([\mu]) = s([\nu]) \in \{[v] : v \in \Lambda_1^0
\setminus \phi_1(\Gamma^0)\}$, then $\psi_1(s_{[\mu]} s^*_{[\nu]}) = s^1_\mu
(s^1_\nu)^*$, so $\psi_1(\ker(\psi_2)) = \ker(\pi_1)$ as claimed.

We have now established the hypotheses of \cite[Proposition~3.1]{Pedersen}, which then
implies that $\TT C^*(\Sigma)$ is a pullback of $\TT C^*(\Lambda_1)$ and $\TT
C^*(\Lambda_2)$ over $\TT C^*(\Gamma)$. The universal property of this pullback therefore
yields a homomorphism $\eta : \TT C^*(\Lambda_1) \oplus_{\TT C^*(\Gamma)} \TT
C^*(\Lambda_2) \to \TT C^*(\Sigma)$ which is inverse to $\theta$.
\end{proof}

Recall from \cite{Sims:CJM06} that if $\Lambda$ is a finitely aligned $k$-graph then a
hereditary set $H \subseteq \Lambda^0$ is \emph{saturated} if whenever $E \subseteq
v\Lambda$ is finite exhaustive and $s(E) \subseteq H$ we have $v \in H$.

\begin{cor}\label{cor:CK pullback}
Let $\Lambda_1$ and $\Lambda_2$ and $\Gamma$ be finitely aligned $k$-graphs with no
sources. Suppose that for $i = 1,2$, we have an injective $k$-graph morphism
$\phi_i : \Gamma \hookrightarrow \Lambda_i$ such that  $\phi_i(\Gamma^0)$ is a
co-hereditary subgraph of $\Lambda_i$ and $H_i := \Lambda^0_i \setminus \phi_i(\Gamma^0)$
is saturated. Let $\pi_i : C^*(\Lambda_i) \to C^*(\Gamma)$ be the homomorphism obtained
from Lemma~\ref{lem:exact sequence}, and form the pullback $C^*$-algebra $C^*(\Lambda_1)
\oplus_{C^*(\Gamma)} C^*(\Lambda_2)$ with respect to $\pi_1$ and $\pi_2$. Let $\phi : \phi_1(\Gamma) \to \phi_2(\Gamma)$ be
the isomorphism $\phi_1(\lambda) \mapsto \phi_2(\lambda)$, and let $\sim_\phi$ be the
equivalence relation of Example~\ref{ex:disjoint_union_equiv}. The
isomorphism $\theta : \TT C^*((\Lambda_1 \sqcup \Lambda_2)/{\sim_\phi}) \to \TT
C^*(\Lambda_1) \oplus_{\TT C^*(\Gamma)} \TT C^*(\Lambda_2)$ of Theorem \ref{thm:pullback} descends to an isomorphism
\[
\tilde\theta : C^*((\Lambda_1 \sqcup \Lambda_2)/{\sim_\phi}) \to
C^*(\Lambda_1) \oplus_{C^*(\Gamma)} C^*(\Lambda_2).
\]
\end{cor}
\begin{proof}
Let $\theta$ be the isomorphism of Theorem~\ref{thm:pullback}, and let $q_0 : \TT C^*(\Lambda_1) \oplus \TT C^*(\Lambda_2) \to C^*(\Lambda_1) \oplus C^*(\Lambda_2)$ be the quotient map. Then $q_0$ restricts to a homomorphism
\[
q : \TT C^*(\Lambda_1) \oplus_{\TT C^*(\Gamma)} \TT C^*(\Lambda_2)
	\to C^*(\Lambda_1) \oplus_{C^*(\Gamma)} C^*(\Lambda_2).
\]
Define $t_{[\lambda]} := q(\theta(s_{[\lambda]}))$ for $\lambda \in \Sigma := (\Lambda_1 \sqcup \Lambda_2)/{\sim_\phi}$. We claim that the $t_{[ \lambda]}$ satisfy relation~(CK). Suppose that $E \subseteq v\Sigma$ is finite exhaustive. If $v \not\in \Gamma$, then $E$ is finite exhaustive in $\Lambda_i$ for some $i$ and then~(CK) follows from~(CK) for $\Lambda_i$. Otherwise, $E \cap \Lambda_i$ is exhaustive for each of $i = 1,2$, and then~(CK) for $\Sigma$ follows from~(CK) for $\Lambda_1$ and $\Lambda_2$. So there is a homomorphism $\tilde\theta$ as claimed, and this $\tilde\theta$ is surjective because $\theta$ is.

To see that $\theta$ is injective, we apply the gauge-invariant uniqueness theorem. Let
$\gamma^i$ denote the gauge action on $C^*(\Lambda_i)$. Then the action $\gamma^1 \oplus
\gamma^2$ of $\T^k$ on $C^*(\Lambda_1) \oplus C^*(\Lambda_2)$ restricts to an action
$\beta$ of $\T^k$ on the subalgebra $C^*(\Lambda_1) \oplus_{C^*(\Gamma)} C^*(\Lambda_2)$.
The gauge action $\gamma$ on $C^*\big((\Lambda_1 \sqcup \Lambda_2)/{\sim_\phi}\big)$ then
satisfies $\beta_z \circ \tilde\theta = \tilde\theta \circ \gamma_z$ for all $z$. Since
each $[v] \in (\Lambda_1 \sqcup \Lambda_2)/{\sim_\phi}$ has a representative $v$ in
either $\Lambda_1^0$ or $\Lambda_2^0$ we have $\tilde\theta(p_{[v]}) \not= 0$ for all
$[v] \in \big((\Lambda_1 \sqcup \Lambda_2)/{\sim_\phi}\big)^0$. So the gauge-invariant
uniqueness theorem \cite[Theorem~3.1]{RaeburnSimsEtAl:JFA04} implies that $\tilde\theta$
is injective.
\end{proof}

For the following result, observe that under the hypotheses of
Corollary~\ref{cor:CK pullback}, the isomorphism $\theta : C^*((\Lambda_1 \sqcup \Lambda_2)/{\sim_\phi}) \to
C^*(\Lambda_1) \oplus_{C^*(\Gamma)} C^*(\Lambda_2)$ determines an inclusion $\tilde\theta
: C^*((\Lambda_1 \sqcup \Lambda_2)/{\sim_\phi}) \to C^*(\Lambda_1) \oplus C^*(\Lambda_2)$
satisfying the formula~\eqref{eq:theta formula}. Write $\pi_i : C^*(\Lambda_i) \to
C^*(\Gamma)$ for the homomorphisms of Lemma~\ref{lem:exact sequence}. These induce
homomorphisms $(\pi_i)_* : K_*(C^*(\Lambda_1)) \oplus K_*(C^*(\Lambda_2)) \to
K_*(C^*(\Gamma))$.

\begin{cor}\label{cor:Meyer-Vietoris}
With the hypotheses of Corollary~\ref{cor:CK pullback}, there is a $6$-term exact
sequence in $K$-theory as follows:
\[
\begin{tikzpicture}[>=stealth, xscale=1.5]
    \node (K0A) at (0,0) {$K_0\big(C^*((\Lambda_1 \sqcup \Lambda_2)/{\sim_\phi})\big)$};
    \node (K0sum) at (4,0) {$K_0(C^*(\Lambda_1)) \oplus K_0(C^*(\Lambda_2))$};
    \node (K0Q) at (8,0) {$K_0(C^*(\Gamma))$};
    \node (K1A) at (8,-2) {$K_1\big(C^*((\Lambda_1 \sqcup \Lambda_2)/{\sim_\phi})\big)$};
    \node (K1sum) at (4,-2) {$K_1(C^*(\Lambda_1)) \oplus K_0(C^*(\Lambda_2))$};
    \node (K1Q) at (0,-2) {$K_1(C^*(\Gamma))$};
    \draw[->] (K0A)--(K0sum) node[pos=0.5, above] {\small$\iota_*$};
    \draw[->] (K0sum)--(K0Q) node[pos=0.5, above] {\small$(\pi_1)_* - (\pi_2)_*$};
    \draw[->] (K0Q)--(K1A);
    \draw[->] (K1A)--(K1sum) node[pos=0.5, above] {\small$\iota_*$};
    \draw[->] (K1sum)--(K1Q) node[pos=0.5, above] {\small$(\pi_1)_* - (\pi_2)_*$};
    \draw[->] (K1Q)--(K0A);
\end{tikzpicture}
\]
\end{cor}
\begin{proof}
This follows directly from Corollary~\ref{cor:CK pullback} and the Meyer-Vietoris exact
sequence for pullback $C^*$-algebras \cite[Theorem~21.5.1]{Blackadar:K-theory}.
\end{proof}

\section{Surfaces and the connected-sum operation}\label{sec:surfaces}

\subsection{Skeletons}

Recall from \cite{HRSW} that for $k \ge 1$, each $k$-graph $\Lambda$ is completely determined by
the $k$-coloured graph $E_\Lambda$ with vertices $\Lambda^0$ and edges $\bigsqcup_{i=1}^k
\Lambda^{e_i}$ coloured with $k$ different colours $c_1, \dots, c_k$ (that is, $\lambda$ has
colour $c_i$ if and only if $d(\lambda) = e_i$), together with the \emph{factorisation
rules} $ef=f'e'$ whenever $e,e' \in \Lambda^{e_i}$, $f,f' \in \Lambda^{e_j}$ and
$ef=f'e'$ in $\Lambda$. This $k$-coloured graph is called the \emph{skeleton} of
$\Lambda$. Conversely, any $k$-coloured graph together with a set of bijections between
$ji$-coloured paths and $ij$-coloured paths for distinct $i,j \leq k$, and satisfying the
associativity condition of \cite[\S4]{HRSW} (the condition is vacuous when $k = 2$)
determines a $k$-graph.

By convention, in a $2$-coloured graph the edges of colour $c_1$ are drawn blue (or solid)
and the edges of colour $c_2$ are drawn red (or dashed).

\subsection{The connected-sum operation}

We aim to prove the following Theorem:

\begin{thm}\label{thm:surfaces}
For each compact 2-dimensional manifold $M$, there is a $2$-graph $\Lambda$ whose topological
realisation $X_\Lambda$ is homeomorphic to $M$.
\end{thm}

In order to prove the Theorem we develop a connected-sum operation on $2$-graphs and
apply the connected-sum to the four basic surfaces. The steps are given in
\ref{star}-\ref{sequence}

\subsubsection{} \label{star}
Let $\Lambda_1$ and $\Lambda_2$ be $2$-graphs. Suppose that $u_i, v_i \in \Lambda_i^0$
have the property that $\Lambda_i u_i  = \{u_i\}$ and $v_i \Lambda_i = \{v_i\}$ for $i =
1,2$. Let $\sim$ be the smallest equivalence relation on $\Lambda_1 \sqcup \Lambda_2$
such that $u_1 \sim u_2$ and $v_1 \sim v_2$. Since there are no morphisms $\alpha_i \in
\Lambda_i\setminus \{u_i, v_i\}$ such that $s(\alpha_i)=u_i$ or $r(\alpha_i)=v_i$, the
relation $\sim$ trivially satisfies properties (\ref{it:sim-d})--(\ref{it:sim-lift}) of
Proposition~\ref{prp:quotient k-graph}, and so we may form the quotient $2$-graph
$(\Lambda_1 \sqcup \Lambda_2)/{\sim}$.

\subsubsection{} \label{squares}
Now suppose that for $i=1,2$ there exist commuting squares $f_i g_i = g'_i f'_i$ in
$\Lambda_i$ with $f_i , f_i' \in \Lambda_i^{e_1}$ and $g_i , g_i' \in \Lambda_i^{e_2}$
such that $r(f_i) = u_i$, $s(g_i) = v_i$, and the vertices $u_i, s(f_i), s(g'_i)$ and
$v_i$ are all distinct. Let $\CC$ be the set of factorisation rules for $E_{(\Lambda_1
\cup \Lambda_2)/{\sim}}$ and specify a new set of factorisation rules $\CC'$ by replacing
$f_1 g_1=g'_1f'_1$ and $f_2g_2= g'_2f'_2$ by $f_1 g_1= g'_2f'_2$ and $f_2g_2 = g'_1f'_1$.
Since this new set of factorisation rules still specifies a range- and source-preserving
bijection between red-blue paths and blue-red paths, and since the associativity
condition of \cite{HRSW} is vacuous when $k = 2$, this is also a valid set of
factorisation rules on $E_{(\Lambda_1 \cup \Lambda_2)/{\sim}}$. Hence Theorems
4.4~and~4.5 of \cite{HRSW} imply that there is a unique $2$-graph $\Lambda_1 \#
\Lambda_2$ with skeleton $E_{(\Lambda_1 \cup \Lambda_2)/{\sim}}$ and factorisation rules
$\CC'$, called the \emph{connected-sum} of $\Lambda_1$ and $\Lambda_2$.

\subsubsection{} \label{above}
Proposition~\ref{prp:topological gluing} implies that $X_{(\Lambda_1 \sqcup
\Lambda_2)/{\sim}}$ is the surface formed by gluing the points $[u_1, 0]$ and $[v_1, 0]$
in $X_{\Lambda_1}$ to the points $[u_2, 0]$ and $[v_2, 0]$ of $X_{\Lambda_2}$. Let
$\alpha,\beta \in \big((\Lambda_1 \sqcup \Lambda_2)/{\sim}\big)^{(1,1)}$ be the elements
$\alpha = f_1g_1 = g'_1f'_1$ and $\beta = f_2g_2=g'_2f'_2$. Likewise, let $\eta = f_1g_1
= g'_2f'_2$ and $\zeta = f_2g_2 = g'_1f'_1$ in $(\Lambda_1 \# \Lambda_2)^{(1,1)}$. Then
we have
\[
X_{\Lambda_1 \# \Lambda_2}
    = X_{(\Lambda_1 \sqcup \Lambda_2)/{\sim}}
        \setminus \big(\{\alpha, \beta\} \times (0, (1,1)) \big)
        \cup \big(\{\eta,\zeta\} \times (0, (1,1))\big).
\]
The effect of this on the topological realisation is illustrated below: the squares
corresponding to $\alpha$ and $\beta$ in $(\Lambda_1 \sqcup \Lambda_2)/{\sim}$ are
illustrated on the left, and those corresponding to $\eta$ and $\zeta$ in $\Lambda_1 \#
\Lambda_2$ are illustrated on the right.
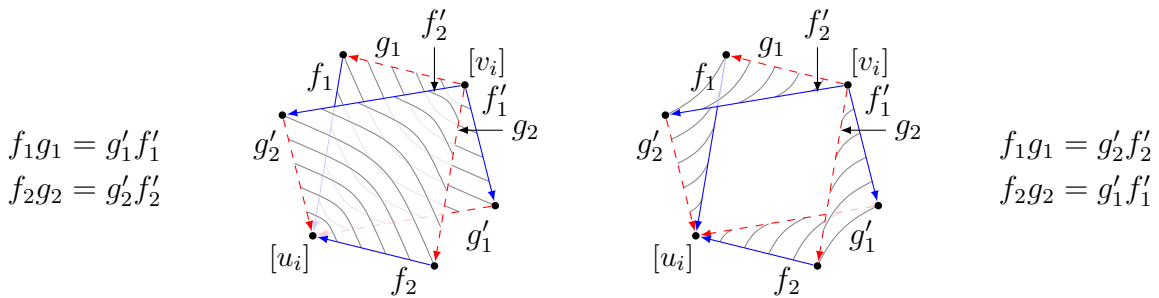
\[
\begin{tikzpicture}[scale=2]
    \foreach \b/\s/\m in {0.24/0.04/0.1, 0.48/0.08/0.2, 0.72/0.12/0.3, 0.96/0.16/0.4, 1.2/0.20/0.5,%
                          1.16/0.36/0.6, 1.12/0.52/0.7, 1.08/0.68/0.8, 1.04/0.84/0.9} {
        \draw[thick, white] (\b,\s) .. controls (\m,\m) .. (\s,\b);
        \draw[very thin, gray] (\b,\s) .. controls (\m,\m) .. (\s,\b);
    }
    \node[inner sep=1pt, circle] (v) at (1,1) {};
    \node[inner sep=1pt, circle] (x1) at (0.8,-0.2) {};
    \node[inner sep=1pt, circle] (y1) at (-0.2,0.8) {}; 
    \node[inner sep=1pt, circle] (y2) at (1.2,0.2) {};
    \node[inner sep=1pt, circle] (x2) at (0.2,1.2) {};
    \node[inner sep=1pt, circle] (u) at (0,0) {};
    \node[anchor=north east, circle, inner sep=0.5pt] at (u.south west) {\small$[u_i]$};
    \node[anchor=south west, circle, inner sep=0.5pt] at (v.north east) {\small$[v_i]$};
    \draw[blue,-latex] (v)--(y2);
    \draw[blue,-latex] (x2)--(u);
    \draw[red,dashed,-latex] (v)--(x2);
    \draw[red,dashed,-latex] (y2)--(u);
    \filldraw[white, fill=white, opacity=0.8] (1,1)--(0.8,-0.2)--(0,0)--(-0.2,0.8)--cycle;
%
    \foreach \b/\s/\m in {0.76/0.96/0.9, 0.54/0.92/0.8, 0.30/0.88/0.7, 0.06/0.84/0.6, -0.2/0.8/0.5,%
                          -0.16/0.64/0.4, -0.12/0.48/0.3, -0.08/0.32/0.2, -0.04/0.16/0.1} {
        \draw[thick, white] (\b,\s) .. controls (\m,\m) .. (\s,\b);
        \draw[very thin, gray] (\b,\s) .. controls (\m,\m) .. (\s,\b);
    }
    \draw[blue,-latex] (v)--(y1);
    \draw[blue,-latex] (x1)--(u);
    \draw[red,dashed,-latex] (v)--(x1);
    \draw[red,dashed,-latex] (y1)--(u);
    \node[inner sep=1pt, circle, fill=black] at (u) {};
    \node[inner sep=1pt, circle, fill=black] at (v) {};
    \node[inner sep=1pt, circle, fill=black] at (x1) {};
    \node[inner sep=1pt, circle, fill=black] at (x2) {};
    \node[inner sep=1pt, circle, fill=black] at (y1) {};
    \node[inner sep=1pt, circle, fill=black] at (y2) {};
\node at (1.1,0) {$g_1'$};
\node at (0.05,1.05) {$f_1$};
\node at (0.5,1.25) {$g_1$};
\node at (-0.3,0.6) {$g_2'$};
\node at (1.2,0.9) {$f_1'$};
\node at (0.6,-0.3) {$f_2$};
\draw[black,-latex] (0.8,1.25)--(0.8,0.95);
\node at (0.8,1.4) {$f_2'$};
\draw[black,-latex] (1.25,0.7)--(0.95,0.7);
\node at (1.4,0.7) {$g_2$};
\node at (-1.5,0.6) {$f_1 g_1 = g_1'f_1'$};
\node at (-1.5,0.3) {$f_2 g_2 = g_2'f_2'$};
\end{tikzpicture}
\hskip1cm
\begin{tikzpicture}[scale=2]
    \node[inner sep=1pt, circle] (v) at (1,1) {};
    \node[inner sep=1pt, circle] (x1) at (0.8,-0.2) {};
    \node[inner sep=1pt, circle] (y1) at (-0.2,0.8) {}; 
    \node[inner sep=1pt, circle] (y2) at (1.2,0.2) {};
    \node[inner sep=1pt, circle] (x2) at (0.2,1.2) {};
    \node[inner sep=1pt, circle] (u) at (0,0) {};
    \node[anchor=north east, circle, inner sep=0.5pt] at (u.south west) {\small$[u_i]$};
    \node[anchor=south west, circle, inner sep=0.5pt] at (v.north east) {\small$[v_i]$};
    \foreach \x/\xx/\y/\yy in {0.04/-0.04/0.24/0.16, 0.08/-0.08/0.48/0.32, 0.12/-0.12/0.72/0.48, 0.16/-0.16/0.96/0.64} {
        \draw[thick, white,in=25,out=245] (\x,\y) to (\xx,\yy);
        \draw[very thin, gray,in=25,out=245] (\x,\y) to (\xx,\yy);
    }
    \foreach \x/\xx/\y/\yy in {0.24/0.16/0.04/-0.04, 0.48/0.32/0.08/-0.08, 0.72/0.48/0.12/-0.12, 0.96/0.64/0.16/-0.16} {
        \draw[thick, white,in=65,out=205] (\x,\y) to (\xx,\yy);
        \draw[very thin, gray,in=65,out=205] (\x,\y) to (\xx,\yy);
    }
    \draw[red,dashed,-latex] (y2)--(u);
    \draw[blue,-latex] (x2)--(u);
    \filldraw[white, fill=white, opacity=0.7] (1,1)--(1.2,0.2)--(0.8,-0.2)--cycle;
    \filldraw[white, fill=white, opacity=0.7] (1,1)--(0.2,1.2)--(-0.2,0.8)--cycle;
    \foreach \x/\xx/\y/\yy in {0.2/-0.2/1.2/0.8, 0.36/0.04/1.16/0.84, 0.52/0.28/1.12/0.88, 0.68/0.52/1.08/0.92, 0.84/0.76/1.04/0.96} {
        \draw[thick, white,in=25,out=245] (\x,\y) to (\xx,\yy);
        \draw[very thin, gray,in=25,out=245] (\x,\y) to (\xx,\yy);
    }
    \foreach \x/\xx/\y/\yy in {1.2/0.8/0.2/-0.2, 1.16/0.84/0.36/0.04, 1.12/0.88/0.52/0.28, 1.08/0.92/0.68/0.52, 1.04/0.96/0.84/0.76} {
        \draw[thick, white,in=65,out=205] (\x,\y) to (\xx,\yy);
        \draw[very thin, gray,in=65,out=205] (\x,\y) to (\xx,\yy);
    }
    \draw[blue,-latex] (v)--(y1);
    \draw[blue,-latex] (v)--(y2);
    \draw[blue,-latex] (x1)--(u);
    \draw[red,dashed,-latex] (v)--(x1);
    \draw[red,dashed,-latex] (v)--(x2);
    \draw[red,dashed,-latex] (y1)--(u);
    \node[inner sep=1pt, circle, fill=black] at (u) {};
    \node[inner sep=1pt, circle, fill=black] at (v) {};
    \node[inner sep=1pt, circle, fill=black] at (x1) {};
    \node[inner sep=1pt, circle, fill=black] at (x2) {};
    \node[inner sep=1pt, circle, fill=black] at (y1) {};
    \node[inner sep=1pt, circle, fill=black] at (y2) {};
\node at (1.1,0) {$g_1'$};
\node at (0.05,1.05) {$f_1$};
\node at (0.5,1.25) {$g_1$};
\node at (-0.3,0.6) {$g_2'$};
\node at (1.2,0.9) {$f_1'$};
\node at (0.6,-0.3) {$f_2$};
\draw[black,-latex] (0.8,1.25)--(0.8,0.95);
\node at (0.8,1.4) {$f_2'$};
\draw[black,-latex] (1.25,0.7)--(0.95,0.7);
\node at (1.4,0.7) {$g_2$};
\node at (2.5,0.6) {$f_1 g_1 = g_2'f_2'$};
\node at (2.5,0.3) {$f_2 g_2 = g_1'f_1'$};
\end{tikzpicture}
\]
The operation described in \ref{above} deletes the interior of the square $f_1 g_1 = g_1'
f_1'$  in $\Lambda_1$ and the square $f_2 g_2 = g_2' f_2'$ in $\Lambda_2$, and inserts a
copy of a unit square bounded by $f_1, g_1, f_2'$ and $g_2'$ and another bounded by
$f_1', g_1', f_2$ and $g_2$. Since the topological realisation of a $2$-graph is obtained
by pasting a unit square into each commuting square and identifying common edges (see the
remark after the proof of Lemma 3.9 in \cite{KKQS}), we have shown that $X_{\Lambda_1 \#
\Lambda_2}$ is homeomorphic to the connected sum $X_{\Lambda_1} \# X_{\Lambda_2}$ of the
topological spaces $X_{\Lambda_1}$ and $X_{\Lambda_2}$.

\subsubsection{}
In the connected-sum $\Lambda_1\# \Lambda_2$, the vertices $u = [u_1]$ and $v = [v_1]$ of $\Lambda_1 \# \Lambda_2$
and the path $f_1g_1 = g'_2f'_2$ in $\Lambda^{(1,1)}$ have the properties required of
$u_1, v_1$ in \ref{star}. So the process we have just described can be iterated.

Examples 3.10--3.13 in Section~3.1 of \cite{KKQS} illustrate finite 2-graphs $\Lambda_S$,
$\Lambda_T$, $\Lambda_K$, $\Lambda_P$ whose topological realisations are the sphere, the
torus, the Klein bottle and the projective plane respectively. Their skeletons are
depicted below, each with the vertices labelled $u$ and $v$ satisfying the conditions
in~\ref{star} and a commuting square $\alpha = ag=ec$ as in~\ref{squares}:
\[
\begin{tikzpicture}[scale=1.5]
    \node[inner sep= 1pt] at (-1.6,0,0) {$\Lambda_S:=$};
    \node[inner sep= 1pt] (100) at (1,0,0) {$v$};
    \node[inner sep= 1pt] (-100) at (-1,0,0) {$y$};
    \node[inner sep= 1pt] (010) at (0,1,0) {$w$};
    \node[inner sep= 1pt] (0-10) at (0,-1,0) {$x$};
    \node[inner sep= 1pt] (001) at (0,0,1) {$u$};
    \node[inner sep= 1pt] (00-1) at (0,0,-1) {$z$};
    \draw[-latex, blue] (100) .. controls +(0,0.6,0) and +(0.6,0,0) .. (010) node[pos=0.5, anchor=south west] {\color{black}$a$};
    \draw[-latex, red, dashed] (100) .. controls +(0,-0.6,0) and +(0.6,0,0) .. (0-10) node[pos=0.5, anchor=north west] {\color{black}$e$};
    \draw[-latex, blue] (-100) .. controls +(0,0.6,0) and +(-0.6,0,0) .. (010) node[pos=0.5, anchor=south east] {\color{black}$b$};
    \draw[-latex, red, dashed] (-100) .. controls +(0,-0.6,0) and +(-0.6,0,0) .. (0-10) node[pos=0.5, anchor=north east] {\color{black}$f$};
    \draw[-latex, red, dashed] (010) .. controls +(0,0,0.6) and +(0,0.6,0) .. (001) node[pos=0.5, anchor=east] {\color{black}$g$};
    \draw[-latex, red, dashed] (010) .. controls +(0,0,-0.6) and +(0,0.6,0) .. (00-1) node[pos=0.8, anchor=east] {\color{black}$h$};
    \draw[-latex, blue] (0-10) .. controls +(0,0,0.6) and +(0,-0.6,0) .. (001) node[pos=0.85, anchor=west] {\color{black}$c$};
    \draw[-latex, blue] (0-10) .. controls +(0,0,-0.6) and +(0,-0.6,0) .. (00-1) node[pos=0.5, anchor=south east] {\color{black}$d$};
\end{tikzpicture}
\hspace{1.5cm}
\begin{tikzpicture}[scale=1.3]
    \node[inner sep= 1pt] at (-2.4,0,0) {$\Lambda_T:=$};
    \node[inner sep= 1pt] (002) at (0,0,2) {$u$};
    \node[inner sep= 1pt] (001) at (0,0,1.2) {$w$};
    \node[inner sep= 1pt] (00-1) at (0,0,-1.2) {$v$};
    \node[inner sep= 1pt] (00-2) at (0,0,-2) {$x$};
    \draw[-latex, red, dashed] (001) .. controls +(0,0.8,0) and +(0,0.8,0) .. (002) node[pos=0.25, anchor=west] {\color{black}$g$};
    \draw[-latex, red, dashed] (001) .. controls +(0,-0.8,0) and +(0,-0.8,0) .. (002) node[pos=0.5, anchor=north] {\color{black}$h$};
    \draw[-latex, red, dashed] (00-1) .. controls +(0,0.8,0) and +(0,0.8,0) .. (00-2) node[pos=0.5, anchor=south] {\color{black}$e$};
    \draw[-latex, red, dashed] (00-1) .. controls +(0,-0.8,0) and +(0,-0.8,0) .. (00-2) node[pos=0.25, anchor=east] {\color{black}$f$};
    \draw[-latex, blue] (00-1) .. controls +(1,0,0) and +(1,0,0) .. (001) node[pos=0.6, anchor=west] {\color{black}$a$};
    \draw[-latex, blue] (00-1) .. controls +(-1,0,0) and +(-1,0,0) .. (001) node[pos=0.4, anchor=east] {\color{black}$b$};
    \draw[-latex, blue] (00-2) .. controls +(2,0,0) and +(2,0,0) .. (002) node[pos=0.6, anchor=west] {\color{black}$c$};
    \draw[-latex, blue] (00-2) .. controls +(-2,0,0) and +(-2,0,0) .. (002) node[pos=0.4, anchor=east] {\color{black}$d$};
\end{tikzpicture}
\]
\[
\begin{tikzpicture}[scale=2.2]
\node[inner sep= 1pt] at (-0.9,0.5) {$\Lambda_K:=$};
\node[inner sep=0.5pt, circle] (u) at (0,0) {$u$};
    \node[inner sep=0.5pt, circle] (v) at (1,0) {$x$};
    \node[inner sep=0.5pt, circle] (w) at (0,1) {$w$};
    \node[inner sep=0.5pt, circle] (x) at (1,1) {$v$};
    \draw[blue,-latex,out=160, in=20] (v) to node[pos=0.5,above, black] {$c$} (u);
    \draw[blue,-latex,out=200, in=340] (v) to node[pos=0.5,below, black] {$d$} (u);
    \draw[blue,-latex,out=160, in=20] (x) to node[pos=0.5,above, black] {$a$} (w);
    \draw[blue,-latex,out=200, in=340] (x) to node[pos=0.5,below, black] {$b$} (w);
    \draw[red, dashed, -latex, out=290, in=70] (w) to node[pos=0.5,right, black] {$h$} (u);
    \draw[red, dashed, -latex, out=250, in=110] (w) to node[pos=0.5,left, black] {$g$} (u);
    \draw[red, dashed, -latex, out=290, in=70] (x) to node[pos=0.5,right, black] {$f$} (v);
    \draw[red, dashed, -latex, out=250, in=110] (x) to node[pos=0.5,left, black] {$e$} (v);
\end{tikzpicture}
\hspace{1.7cm}
\begin{tikzpicture}[scale=1.7]
\node[inner sep= 1pt] at (-1.6,0) {$\Lambda_P:=$};
\node[inner sep=.5pt, circle] (u) at (0,0) {$u$};
    \node[inner sep=.5pt, circle] (v) at (0,1) {$y$};
    \node[inner sep=.5pt, circle] (w) at (0,-1) {$w$};
    \node[inner sep=.5pt, circle] (x) at (1,0) {$x$};
    \node[inner sep=.5pt, circle] (y) at (-1,0) {$v$};
    \draw[-latex, blue] (v) .. controls +(-0.2,-0.5) .. (u) node[pos=0.5, left, black] {$c$};
    \draw[-latex, blue] (v) .. controls +(0.2,-0.5) .. (u) node[pos=0.5, right, black] {$d$};
    \draw[-latex, red, dashed] (w) .. controls +(-0.2,0.5) .. (u) node[pos=0.5, left, black] {$g$};
    \draw[-latex, red, dashed] (w) .. controls +(0.2,0.5) .. (u) node[pos=0.5, right, black] {$h$};
    \draw[-latex, blue] (x)--(w) node[pos=0.5, anchor=north west, black] {$b$};
    \draw[-latex, red, dashed] (x)--(v) node[pos=0.5, anchor=south west, black] {$f$};
    \draw[-latex, blue] (y)--(w) node[pos=0.5, anchor=north east, black] {$a$};
    \draw[-latex, red, dashed] (y)--(v) node[pos=0.5, anchor=south east, black] {$e$};
\end{tikzpicture}
\]

\subsubsection{}\label{sequence} Fix a sequence $\Lambda_i$ of $2$-graphs, each of which is a copy of one of $\Lambda_S$, $\Lambda_T$, $\Lambda_K$ or $\Lambda_P$. Let $\Gamma_1 = \Lambda_1$ and inductively construct $\Gamma_i
= \Gamma_{i-1} \# \Lambda_i$ where the connected-sum construction is applied to $\alpha_b
\in \Gamma_{i-1}$ and $\alpha_a \in \Lambda_i$. Then $\bigcup^\infty_{i=1}
\bigcap^\infty_{j=i} \Gamma_i$ is a $2$-graph and its topological realisation is the
connected-sum of the $\Lambda_i$. So we can form any countable connected-sum of these
four surfaces as the topological realisation of a $2$-graph.

\begin{proof}[Proof of Theorem~\ref{thm:surfaces}]
The classification of compact 2-dimensional manifolds (see for example
\cite[Theorem~I.7.2]{Massey}) says that any such object is a sphere, a connected sum of
$n$ $2$-tori, or the connected sum of the latter with either the Klein bottle or the
projective plane. Fix such a decomposition, and then apply the proceedure
\ref{star}--\ref{sequence} to the corresponding collection of $2$-graphs. The topological
realisation of the resulting $2$-graph is then the desired surface.
\end{proof}

\begin{ex}
Let $\Lambda_i = \Lambda_{T}$, for $i=1,2$  where $\Lambda_T$ is the $2$-graph whose
topological realisation is the $1$-holed torus. The skeletons of the $\Lambda_i$ both have the
form of the skeleton on the left
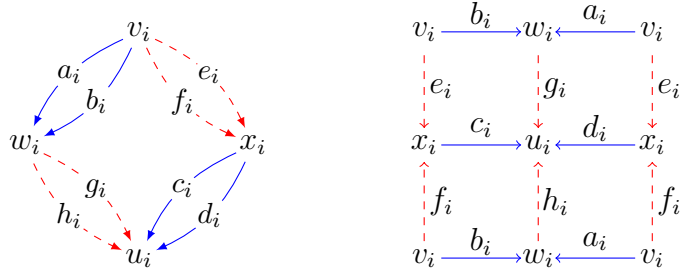
\[
\begin{tikzpicture}[scale=1.5]
\begin{scope}[xshift=0cm]
    \node[inner sep=0.5pt, circle] (x) at (0,2) {$v_i$};
    \node[inner sep=0.5pt, circle] (v) at (1,1) {$x_i$};
    \node[inner sep=0.5pt, circle] (w) at (-1,1) {$w_i$};
    \node[inner sep=0.5pt, circle] (u) at (0,0) {$u_i$};
    \draw[-latex, blue, out=205, in=65] (x) to node[black, pos=0.5,circle,inner sep=0.5pt,fill=white] {\small$a_i$} (w);
    \draw[-latex, blue, out=245, in=25] (x) to node[black, pos=0.5,circle,inner sep=0.5pt,fill=white] {\small$b_i$} (w);
    \draw[-latex, red, dashed, out=335, in=115] (x) to node[black, pos=0.5,circle,inner sep=0.5pt,fill=white] {\small$e_i$} (v);
    \draw[-latex, red, dashed, out=295, in=155] (x) to node[black, pos=0.5,circle,inner sep=0.5pt,fill=white] {\small$f_i$} (v);
    \draw[-latex, blue, out=205, in=65] (v) to node[black, pos=0.5,circle,inner sep=0.5pt,fill=white] {\small$c_i$} (u);
    \draw[-latex, blue, out=245, in=25] (v) to node[black, pos=0.5,circle,inner sep=0.5pt,fill=white] {\small$d_i$} (u);
    \draw[-latex, red, dashed, out=335, in=115] (w) to node[black, pos=0.5,circle,inner sep=0.5pt,fill=white] {\small$g_i$} (u);
    \draw[-latex, red, dashed, out=295, in=155] (w) to node[black, pos=0.5,circle,inner sep=0.5pt,fill=white] {\small$h_i$} (u);
\end{scope}
\begin{scope}[xshift=2.5cm]
\draw[->,blue] (0.15,0) -- (0.85,0);
\draw[<-,blue] (1.15,0) -- (1.85,0);
\draw[->,blue] (0.15,1) -- (0.85,1);
\draw[<-,blue] (1.15,1) -- (1.85,1);
\draw[->,blue] (0.15,2) -- (0.85,2);
\draw[<-,blue] (1.15,2) -- (1.85,2);
\draw[->,red, dashed] (0,0.15) -- (0,0.85);
\draw[->,red, dashed] (1,0.15) -- (1,0.85);
\draw[->,red, dashed] (2,0.15) -- (2,0.85);
\draw[<-,red, dashed] (0,1.15) -- (0,1.85);
\draw[<-,red, dashed] (1,1.15) -- (1,1.85);
\draw[<-,red, dashed] (2,1.15) -- (2,1.85);
\node at (0,0) {$v_i$};
\node at (0,1) {$x_i$};
\node at (0,2) {$v_i$};
\node at (1,0) {$w_i$};
\node at (1,1) {$u_i$};
\node at (1,2) {$w_i$};
\node at (2,0) {$v_i$};
\node at (2,1) {$x_i$};
\node at (2,2) {$v_i$};
\node at (0.5,0.15) {$b_i$};
\node at (1.5,0.15) {$a_i$};
\node at (0.5,1.15) {$c_i$};
\node at (1.5,1.15) {$d_i$};
\node at (0.5,2.15) {$b_i$};
\node at (1.5,2.15) {$a_i$};
\node at (0.15,0.5) {$f_i$};
\node at (0.15,1.5) {$e_i$};
\node at (1.15,0.5) {$h_i$};
\node at (1.15,1.5) {$g_i$};
\node at (2.15,0.5) {$f_i$};
\node at (2.15,1.5) {$e_i$};
\end{scope}
\end{tikzpicture}
\]
with factorisation rules given by the square commuting diagrams on the right. To apply
our construction we set $u=[u_i]$ and $v=[v_i]$ with the
distinguished squares $d_1 f_1 = h_1 a_1$ of $\Lambda_1$ and $c_2 e_2 = g_2 b_2$ of
$\Lambda_{2}$. Then the skeleton of the connected sum $\Lambda_1 \# \Lambda_2$ has the
form of the diagram on the left with the same factorisation rules as above except that
$d_1f_1 = g_2b_2$ and $h_1a_1 = c_2e_2$.
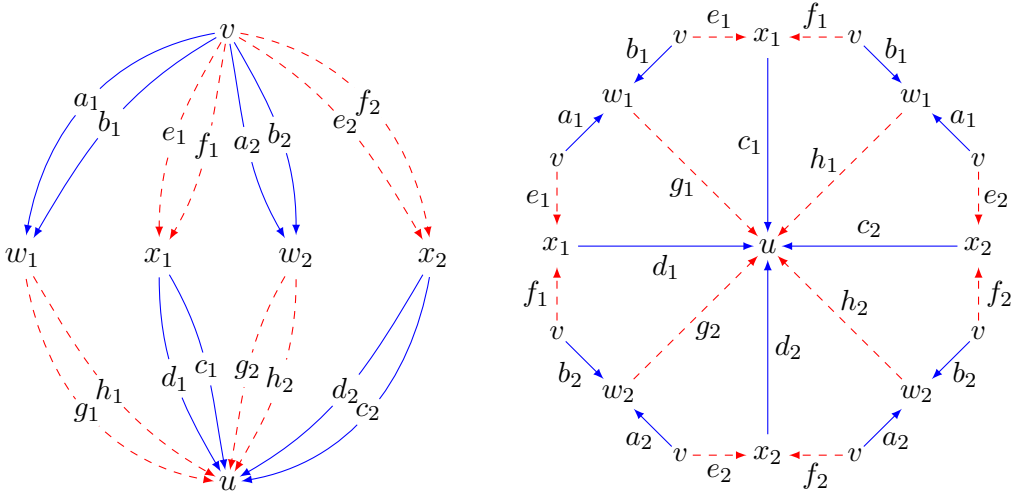
\[
\begin{tikzpicture}[scale=1.5]
    \node[inner sep=0.5pt, circle] (x) at (0,2) {$v$};
    \node[inner sep=0.5pt, circle] (w1) at (-1.8,0) {$w_1$};
    \node[inner sep=0.5pt, circle] (v1) at (-0.6,0) {$x_1$};
    \node[inner sep=0.5pt, circle] (w2) at (0.6,0) {$w_2$};
    \node[inner sep=0.5pt, circle] (v2) at (1.8,0) {$x_2$};
    \node[inner sep=0.5pt, circle] (u) at (0,-2) {$u$};
    \draw[-latex, blue, out=190, in=80] (x) to node[black, pos=0.5,circle,inner sep=0.5pt,fill=white] {\small$a_1$} (w1);
    \draw[-latex, blue, out=210, in=60] (x) to node[black, pos=0.5,circle,inner sep=0.5pt,fill=white] {\small$b_1$} (w1);
    \draw[-latex, blue, out=280, in=120] (x) to node[black, pos=0.5,circle,inner sep=0.5pt,fill=white] {\small$a_2$} (w2);
    \draw[-latex, blue, out=300, in=90] (x) to node[black, pos=0.5,circle,inner sep=0.5pt,fill=white] {\small$b_2$} (w2);
    \draw[-latex, red, dashed, out=240, in=90] (x) to node[black, pos=0.5,circle,inner sep=0.5pt,fill=white] {\small$e_1$} (v1);
    \draw[-latex, red, dashed, out=260, in=60] (x) to node[black, pos=0.5,circle,inner sep=0.5pt,fill=white] {\small$f_1$} (v1);
    \draw[-latex, red, dashed, out=330, in=120] (x) to node[black, pos=0.5,circle,inner sep=0.5pt,fill=white] {\small$e_2$} (v2);
    \draw[-latex, red, dashed, out=350, in=100] (x) to node[black, pos=0.5,circle,inner sep=0.5pt,fill=white] {\small$f_2$} (v2);
    \draw[-latex, blue, in=10, out=260] (v2) to node[black, pos=0.5,circle,inner sep=0.5pt,fill=white] {\small$c_2$} (u);
    \draw[-latex, blue, in=30, out=240] (v2) to node[black, pos=0.5,circle,inner sep=0.5pt,fill=white] {\small$d_2$} (u);
    \draw[-latex, blue, in=100, out=300] (v1) to node[black, pos=0.5,circle,inner sep=0.5pt,fill=white] {\small$c_1$} (u);
    \draw[-latex, blue, in=120, out=270] (v1) to node[black, pos=0.5,circle,inner sep=0.5pt,fill=white] {\small$d_1$} (u);
    \draw[-latex, red, dashed, in=60, out=270] (w2) to node[black, pos=0.5,circle,inner sep=0.5pt,fill=white] {\small$h_2$} (u);
    \draw[-latex, red, dashed, in=80, out=240] (w2) to node[black, pos=0.5,circle,inner sep=0.5pt,fill=white] {\small$g_2$} (u);
    \draw[-latex, red, dashed, in=170, out=280] (w1) to node[black, pos=0.5,circle,inner sep=0.5pt,fill=white] {\small$g_1$} (u);
    \draw[-latex, red, dashed, in=150, out=300] (w1) to node[black, pos=0.5,circle,inner sep=0.5pt,fill=white] {\small$h_1$} (u);
\end{tikzpicture} \qquad
\begin{tikzpicture}[scale=1.5]
    \node[inner sep=0.5pt, circle] (u) at (0,0) {$u$};
    \foreach \n/\T in {1/22.5, 2/67.5, 3/112.5, 4/157.5, 5/202.5, 6/247.5, 7/292.5, 8/337.5} {
        \node[inner sep=0.5pt, circle] (x\n) at (\T:2) {\small$v$};
    }
    \foreach \n/\nn in {1/2, 2/3, 3/4, 4/5, 5/6, 6/7, 7/8, 8/1} {
        \draw[white] (x\n)--(x\nn) node[pos=0.5, inner sep=0.5pt,circle] (z\n) {\phantom{\small$w_0$}};
    }
    \node at (z1) {\small$w_1$};
    \node at (z2) {\small$x_1$};
    \node at (z3) {\small$w_1$};
    \node at (z4) {\small$x_1$};
    \node at (z5) {\small$w_2$};
    \node at (z6) {\small$x_2$};
    \node at (z7) {\small$w_2$};
    \node at (z8) {\small$x_2$};
    \foreach \n/\nn in {2/1, 4/3, 6/5, 8/7} {
        \draw[-latex, blue] (x\n)--(z\nn) node[pos=0.5, inner sep=0] (z\nn x\n) {};
        \draw[-latex, blue] (x\nn)--(z\nn) node[pos=0.5, inner sep=0] (z\nn x\nn) {};
        \draw[-latex, red, dashed] (z\nn)--(u) node[pos=0.5, inner sep=0] (red\nn) {};
    }
    \foreach \n/\nn in {1/8, 3/2, 5/4, 7/6} {
        \draw[-latex, red, dashed] (x\n)--(z\nn) node[pos=0.5, inner sep=0] (z\nn x\n) {};
        \draw[-latex, red, dashed] (x\nn)--(z\nn) node[pos=0.5, inner sep=0] (z\nn x\nn) {};
        \draw[-latex, blue] (z\nn)--(u) node[pos=0.5, inner sep=0] (blue\nn) {};
    }
    \node[anchor=south east, circle, inner sep=0.5pt] at (red1) {\small$h_1$};
    \node[anchor=north east, circle, inner sep=0.5pt] at (red3) {\small$g_1$};
    \node[anchor=north west, circle, inner sep=0.5pt] at (red5) {\small$g_2$};
    \node[anchor=south west, circle, inner sep=0.5pt] at (red7) {\small$h_2$};
    \node[anchor=east, circle, inner sep=0.5pt] at (blue2) {\small$c_1$};
    \node[anchor=north, circle, inner sep=0.5pt] at (blue4) {\small$d_1$};
    \node[anchor=west, circle, inner sep=0.5pt] at (blue6) {\small$d_2$};
    \node[anchor=south, circle, inner sep=0.5pt] at (blue8) {\small$c_2$};
    \node[anchor=east, circle, inner sep=0.5pt] at (z4x4) {\small$e_1$};
    \node[anchor=east, circle, inner sep=0.5pt] at (z4x5) {\small$f_1$};
    \node[anchor=north east, circle, inner sep=0.5pt] at (z5x5) {\small$b_2$};
    \node[anchor=north east, circle, inner sep=0.5pt] at (z5x6) {\small$a_2$};
    \node[anchor=north, circle, inner sep=0.5pt] at (z6x6) {\small$e_2$};
    \node[anchor=north, circle, inner sep=0.5pt] at (z6x7) {\small$f_2$};
    \node[anchor=north west, circle, inner sep=0.5pt] at (z7x7) {\small$a_2$};
    \node[anchor=north west, circle, inner sep=0.5pt] at (z7x8) {\small$b_2$};
    \node[anchor=west, circle, inner sep=0.5pt] at (z8x8) {\small$f_2$};
    \node[anchor=west, circle, inner sep=0.5pt] at (z8x1) {\small$e_2$};
    \node[anchor=south west, circle, inner sep=0.5pt] at (z1x1) {\small$a_1$};
    \node[anchor=south west, circle, inner sep=0.5pt] at (z1x2) {\small$b_1$};
    \node[anchor=south, circle, inner sep=0.5pt] at (z2x2) {\small$f_1$};
    \node[anchor=south, circle, inner sep=0.5pt] at (z2x3) {\small$e_1$};
    \node[anchor=south east, circle, inner sep=0.5pt] at (z3x3) {\small$b_1$};
    \node[anchor=south east, circle, inner sep=0.5pt] at (z3x4) {\small$a_1$};
\end{tikzpicture}
\]
Organising this skeleton into the commuting diagram on the right (note that this is not
the skeleton of the 2-graph, because some edges of the 2-graph appear more than once in
the diagram), we recognise the standard octohedral planar diagram for a two-holed
2-torus.
\end{ex}

\section{Simplices and $k$-spheres from $k$-graphs}\label{sec:spheres}

In this section we show how to realise a $k$-simplex as the topological
realisation of a $k$-graph, $\Sigma_k$. Combining this with the results of
Section~\ref{sec:quotients} we prove the following result.

\begin{thm}\label{thm:spheres realised}
For each $k  \ge 0$ there is a finite $k$-graph $\Gamma$ whose topological realisation is
homeomorphic to a $k$-sphere.
\end{thm}

Recall our convention that $\mathbf{1}_0 = 0$, the unique element of $\N^0$. We begin with
the construction of $\Sigma_k$.

\begin{dfn}
A function $f : \{0, \dots, k\} \to \{0, \dots, k\}$ is a
\emph{$k$-placing}, or just a \emph{placing} if
\begin{equation}\label{eq:placing}
f(j) = \big|\{i : f(i) < f(j)\}\big|\quad\text{ for all $j \le k$}.
\end{equation}
We write $P_k$ for set of $k$-placings, and for $f,g \in P_k$, we write $f \le g$ if
$f(i) \le g(i)$ for all $i$.
\end{dfn}
Every permutation of $\{0, \dots, k\}$ is a placing, and these are precisely
the maximal placings with respect to $\le$. We write $P^{\max}_k$ for set of maximal
placings. The zero function $0 : i \mapsto 0$ is the unique minimum $k$-placing.

We will construct a $k$-graph whose vertices are the $k$-placings. (An acknowledgement is
in order: we found this parameterisation using The On-Line Encyclopedia of Integer
Sequences \cite{oeis}---we constructed the first three examples by hand, and then used
the OEIS to seach for the sequence of the numbers of vertices appearing in these graphs.)
To define the degree map on this $k$-graph, we first introduce the following height
function on placings.

\begin{dfn}
The \emph{height function} $h : P_k \to \N^k$ is given by
\[
h(f)_i = \begin{cases}
    1 &\text{ if $f^{-1}(i) \not= \emptyset$}\\
    0 &\text{ otherwise}
\end{cases}
\qquad\text{ for $1 \le i \le k$,}
\]
and satisfies $h(f) \leq \mathbf{1}_k$ for all $f \in P_k$.
\end{dfn}

Note that $h(0) = 0$, and if $\sigma \in P^{\max}_k$, then $h(\sigma) = \mathbf{1}_k$. We
have $h(f) = e_i$ if and only if $\range(f) = \{0,i\}$, which in turn is equivalent to
$|f^{-1}(0)| = i$ and $|f^{-1}(i)| = k+1-i$.

\begin{prop}\label{prp:k-graph}
Let $\Sigma_k = \{(f,g) \in P_k \times P_k : f \le g\}$.
Define $r(f,g) = (f,f)$, $s(f,g) = (g,g)$, $(f,g)(g,h) = (f,h)$, and $d(f,g) = h(g) -
h(f)$. With respect to these structure maps, $(\Sigma_k, d)$ is a $k$-graph, and $d$ maps
$\Sigma_k$ onto $\{n \in \N^k : n \le \mathbf{1}_k \}$.
\end{prop}

To prove the proposition, we set aside a technical lemma.

\begin{lem}\label{lem:tailfactor}
Suppose $f \in P_k$ and $h(f) \ge z \in \N^k$. Define $g : \{0, \dots, k\} \to \{0,
\dots, k\}$ by
\[
g(i) = \max\{ j \le f(i) : z_j = 1\}.
\]
Then $g \in P_k$, $g \le f$ and $h(g) = z$. Moreover, $g$ is the unique such element of
$P_k$,  
and for $n \in \range(g)$ we have $g^{-1}(\{1, \dots, n-1\}) = f^{-1}(\{1, \dots,
n-1\})$.
\end{lem}
\begin{proof}
It is clear that $g \le f$ as functions.

For $i \le k$, we claim that $\{j : g(j) < g(i)\} = \{j : f(j) < g(i)\}$. Since $g \le
f$, we certainly have $\{j : g(j) < g(i)\} \supseteq \{j : f(j) < g(i)\}$. For the
reverse, we suppose that $f(j) \ge g(i)$ and show that $g(j) \ge g(i)$. We have $z_{g(i)}
= 1$ by definition of $g$. So $f(j) \ge g(i)$ implies that
\[
g(j)
    = \max\{l \le f(j) : z_l = 1\}\\
    \ge \max\{l \le g(i) : z_l = 1\}
    = g(i).
\]
This proves the claim. Now to see that $g$ is a placing, fix $i \le k$. There exists $l
\le k$ such that $f(l) = g(i)$. Hence
\[
g(i)
    = f(l)
    = \big|\{j : f(j) < f(l)\}\big|
    = \big|\{j : f(j) < g(i)\}\big|
    = \big|\{j : g(j) < g(i)\}\big|.
\]
So $g$ is a placing. To see that $h(g) = z$, observe that by definition of $g$ we have
$g^{-1}(i) \not= \emptyset$ if and only if  $z_i =1$.

Suppose that $g' \le f$ and $h(g') = z$. Then for $j \le k$ we have $j \in \range(g')
\iff  z_j = 1$. Since $g \le f$, we have $f^{-1}(\{0, \dots, n\}) \subseteq
(g')^{-1}(\{0, \dots, n\})$ for each $n \le k$. Suppose for contradiction that
$f^{-1}(\{0, \dots, n-1\}) \subsetneq (g')^{-1}(\{0, \dots, n-1\})$ for some $n\in
\range(g')$. Then there is a least such $n$; say $f(l) = g'(l') = n$. Then
\begin{align*}
n &= f(l)
    = \big|\{j : f(j) < f(l)\}\big|
    = \big|f^{-1}(\{0, \dots, n-1\})\big|\\
    &< \big|(g')^{-1}(\{0, \dots, n-1\})\big|
    = \big|\{j : g'(j) < g'(l')\}\big|
    = n,
\end{align*}
a contradiction. So
\begin{equation}\label{eq:preimages}
f^{-1}(\{0, \dots, n-1\}) = (g')^{-1}(\{0, \dots, n-1\})\quad\text{ whenever $n \in \range(g')$.}
\end{equation}
Let $n \in \range(g')$ and let $p = \max\{g'(j) : g'(j) < n\}$. Then
\begin{align*}
(g')^{-1}(\{0, \dots, p-1\}) \sqcup (g')^{-1}(\{p\})
    &= (g')^{-1}(\{0, \dots, n-1\})\\
    &= f^{-1}(\{0, \dots, n-1\})\\
    &= f^{-1}(\{0, \dots, p-1\}) \sqcup f^{-1}(\{i : p \le i < n, \}).
\end{align*}

So for $p \in \range(g')$ we have
\[
    (g')^{-1}(p) = f^{-1}\big(\big\{i \in \range(f) :
                    p = \max\{j \in \range(g') : j \le i\}\big\}\big).
\]
Since $i \in \range(g')$ if and only if  $z_i = 1$, we deduce that $g'(j) = \max\{f(l) :
z_j = 1\}$ for all $j$; that is, $g' = g$. The final assertion now follows
from~\eqref{eq:preimages}.
\end{proof}

\begin{proof}[Proof of Proposition~\ref{prp:k-graph}]
Since $\le$ is a partial order, $\Sigma_k$ forms a category with identity morphisms
$(f,f)$. It is clear that $d$ is a functor. So we just need to check the factorisation
property.

Suppose that $(f', f) \in \Sigma_k$ and $d(f', f) = n+m$. Let $g$ be the placing
constructed from $f$ with $z=h(f') + n$ as in Lemma~\ref{lem:tailfactor}. The uniqueness
assertion in Lemma~\ref{lem:tailfactor} implies that
\[
f'(i) = \max\{j < f(i) :  h(f')_i  = 1\} \le  \max\{j < f(i) :  h(f')_i + n_i = 1 \} = g(i)
\]
So we have $(f',g), (g,f) \in \Sigma_k$ with $d(f',g) = n$ and $d(g,f) = m$, and
$(f',g)(g,f) = n+m$. Uniqueness of this factorisation follows from the uniqueness
assertion in Lemma~\ref{lem:tailfactor}.

The range of the degree map is contained in $\{n \in \N^k : n \le \mathbf{1}_k \}$
because $h(f) \le \mathbf{1}_k$ for all $f$. Each $f \in P^{\max}_k$ satisfies $h(f) =
\mathbf{1}_k$, and so $d(0,f) = \mathbf{1}_k$. Then the factorisation property implies
that the range of $d$ is all of $\{n \in \N^k : n \le \mathbf{1}_k \}$.
\end{proof}

\begin{ex}
On the left side of the following picture, the $1$-skeleton of the $2$-simplex $\Sigma_2$
is shown, with vertices labelled as placing functions, where for example $\{2,01\}$
denotes the placing function $f_{\{2,01\}}$ defined by $f_{\{2,01\}}(2)=0$,
$f_{\{2,01\}}(0)=1$ and $f_{\{2,01\}}(1)=1$. On the right hand side is one of the four
faces of the one skeleton of the tetrahedron-shaped $3$-simplex $\Sigma_3$. In a
$3$-graph it is customary to make the third colour in the $1$-skeleton green (or dotted).
A coloured circle around a vertex indicates an edge of the same colour from that vertex
to the central vertex (not shown) labelled $\{0123\}$. The circles vertices are labelled
by placing functions with exactly first and second place specified with circle colour
blue (solid) if there is a unique first place, red (dashed) if there are ties for first
and second place, and green (dotted) if there is a unique third place. For example, there
is a blue edge from $\{0,123\}$ to $\{0123\}$, a red edge from $\{20,13\}$ to $\{0123\}$,
and a green edge from $\{012,3\}$ to $\{0123\}$.
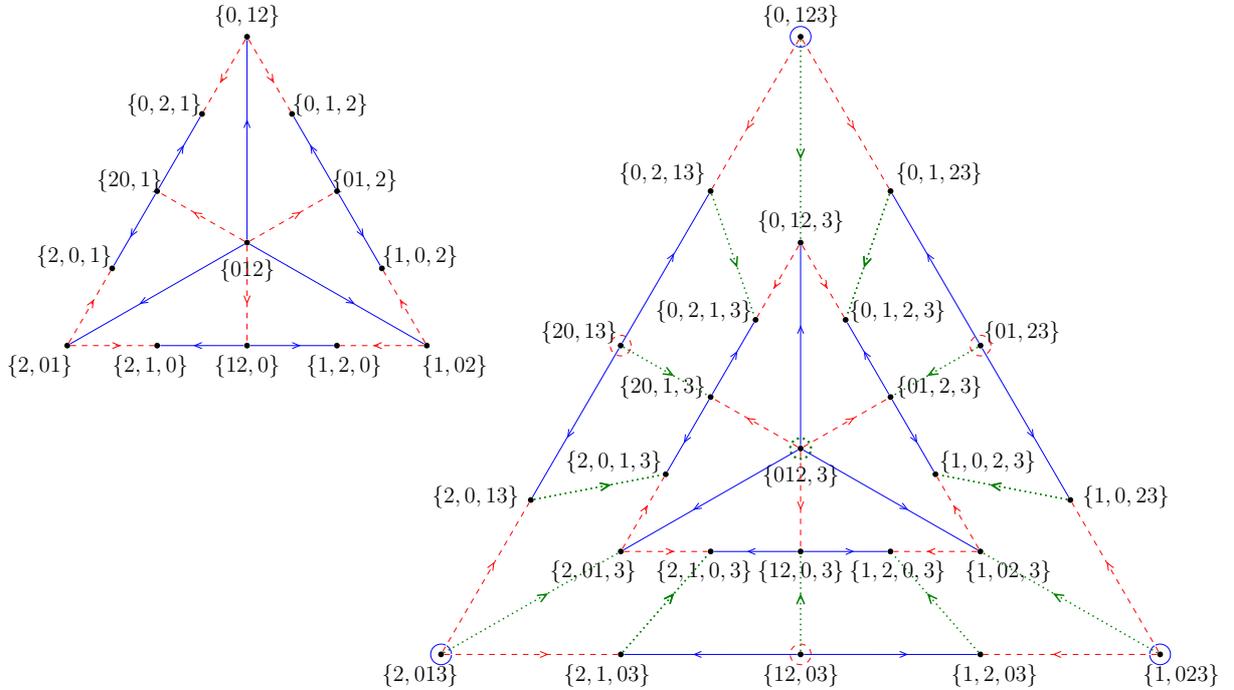
\[
\scalebox{0.70}{
\begin{tikzpicture}[scale=1.3]
    \begin{scope}[decoration={markings,mark=at position 0.6 with {\arrow{angle 45}}},yshift=3cm]
    \node[inner sep=0.5pt, circle,fill=black] (0) at (0,0) {\,};
    \node at (0,-0.4) {$\{012\}$};
    \node[inner sep=0.5pt, circle,fill=black] (a0) at (0,3) {\,};
    \node at (0,3.3) {$\{0,12\}$};
    \node[inner sep=0.5pt, circle,fill=black] (c0) at (-2.598,-1.5) {\,};
    \node at (-3.0,-1.8) {$\{2,01\}$};
    \node[inner sep=0.5pt, circle,fill=black] (b0) at (2.598,-1.5) {\,};
    \node at (3.0,-1.8) {$\{1,02\}$};
    \node[inner sep=0.5pt, circle,fill=black] (a1) at (0,-1.5) {\,};
    \node at (0,-1.8) {$\{12,0\}$};
    \node[inner sep=0.5pt, circle,fill=black] (c1) at (1.3,0.75) {\,};
    \node at (1.7,0.9) {$\{01,2\}$};
    \node[inner sep=0.5pt, circle,fill=black] (b1) at (-1.3,0.75) {\,};
    \node at (-1.7,0.9) {$\{20,1\}$};
    \draw[white] (a0)--(c1) node[pos=0.5, inner sep=0.5pt,circle,fill=black] (a0c1) {\,};
    \node at (1.2,2.0) {$\{0,1,2\}$};
    \draw[white] (c1)--(b0) node[pos=0.5, inner sep=0.5pt,circle,fill=black] (b0c1) {\,};
    \node at (2.5,-0.2) {$\{1,0,2\}$};
    \draw[white] (b0)--(a1) node[pos=0.5, inner sep=0.5pt,circle,fill=black] (a1b0) {\,};
    \node at (1.4,-1.8) {$\{1,2,0\}$};
    \draw[white] (a1)--(c0) node[pos=0.5, inner sep=0.5pt,circle,fill=black] (a1c0) {\,};
    \node at (-1.4,-1.8) {$\{2,1,0\}$};
    \draw[white] (c0)--(b1) node[pos=0.5, inner sep=0.5pt,circle,fill=black] (b1c0) {\,};
    \node at (-2.5,-0.2) {$\{2,0,1\}$};
    \draw[white] (b1)--(a0) node[pos=0.5, inner sep=0.5pt,circle,fill=black] (a0b1) {\,};
    \node at (-1.2,2.0) {$\{0,2,1\}$};
    \draw[red, dashed,postaction={decorate}] (a0)--(a0c1) node[pos=0.5, inner sep=0.5pt, circle] {$\,$};
    \draw[red, dashed,postaction={decorate}] (a0)--(a0b1) node[pos=0.5, inner sep=0.5pt, circle] {$\,$};
    \draw[red, dashed,postaction={decorate}] (b0)--(b0c1) node[pos=0.5, inner sep=0.5pt, circle] {$\,$};
    \draw[red, dashed,postaction={decorate}] (b0)--(a1b0) node[pos=0.5, inner sep=0.5pt, circle] {$\,$};
    \draw[red, dashed,postaction={decorate}] (c0)--(a1c0) node[pos=0.5, inner sep=0.5pt, circle] {$\,$};
    \draw[red, dashed,postaction={decorate}] (c0)--(b1c0) node[pos=0.5, inner sep=0.5pt, circle] {$\,$};
    \draw[red, dashed,postaction={decorate}] (0)--(a1) node[pos=0.5, inner sep=0.5pt, circle] {$\,$};
    \draw[red, dashed,postaction={decorate}] (0)--(b1) node[pos=0.5, inner sep=0.5pt, circle] {$\,$};
    \draw[red, dashed,postaction={decorate}] (0)--(c1) node[pos=0.5, inner sep=0.5pt, circle] {$\,$};
    \draw[blue,postaction={decorate}] (c1)--(a0c1) node[pos=0.5, inner sep=0.5pt, circle] {$\,$};
    \draw[blue,postaction={decorate}] (c1)--(b0c1) node[pos=0.5, inner sep=0.5pt, circle] {$\,$};
    \draw[blue,postaction={decorate}] (a1)--(a1b0) node[pos=0.5, inner sep=0.5pt, circle] {$\,$};
    \draw[blue,postaction={decorate}] (a1)--(a1c0) node[pos=0.5, inner sep=0.5pt, circle] {$\,$};
    \draw[blue,postaction={decorate}] (b1)--(b1c0) node[pos=0.5, inner sep=0.5pt, circle] {$\,$};
    \draw[blue,postaction={decorate}] (b1)--(a0b1) node[pos=0.5, inner sep=0.5pt, circle] {$\,$};
    \draw[blue,postaction={decorate}] (0)--(a0) node[pos=0.5, inner sep=0.5pt, circle] {$\,$};
    \draw[blue,postaction={decorate}] (0)--(b0) node[pos=0.5, inner sep=0.5pt, circle] {$\,$};
    \draw[blue,postaction={decorate}] (0)--(c0) node[pos=0.5, inner sep=0.5pt, circle] {$\,$};
    \end{scope}
    \begin{scope}[decoration={markings,mark=at position 0.6 with {\arrow{angle 45}}},xshift=8cm]
    \node[inner sep=0.5pt, circle,fill=black] (0) at (0,0) {\,};
    \node at (0,-0.4) {$\{012,3\}$};
    \draw[green!50!black, dotted, very thick] (0,0) circle (0.15cm);
    \node[inner sep=0.5pt, circle,fill=black] (a0) at (0,3) {\,};
    \node at (0,3.3) {$\{0,12,3\}$};
    \node[inner sep=0.5pt, circle,fill=black] (c0) at (-2.598,-1.5) {\,};
    \node at (-3.0,-1.8) {$\{2,01,3\}$};
    \node[inner sep=0.5pt, circle,fill=black] (b0) at (2.598,-1.5) {\,};
    \node at (3.0,-1.8) {$\{1,02,3\}$};
    \node[inner sep=0.5pt, circle,fill=black] (a1) at (0,-1.5) {\,};
    \node at (0,-1.8) {$\{12,0,3\}$};
    \node[inner sep=0.5pt, circle,fill=black] (c1) at (1.3,0.75) {\,};
    \node at (2.0,0.9) {$\{01,2,3\}$};
    \node[inner sep=0.5pt, circle,fill=black] (b1) at (-1.3,0.75) {\,};
    \node at (-2.0,0.9) {$\{20,1,3\}$};
    \draw[white] (a0)--(c1) node[pos=0.5, inner sep=0.5pt,circle,fill=black] (a0c1) {\,};
    \node at (1.4,2.0) {$\{0,1,2,3\}$};
    \draw[white] (c1)--(b0) node[pos=0.5, inner sep=0.5pt,circle,fill=black] (b0c1) {\,};
    \node at (2.7,-0.2) {$\{1,0,2,3\}$};
    \draw[white] (b0)--(a1) node[pos=0.5, inner sep=0.5pt,circle,fill=black] (a1b0) {\,};
    \node at (1.4,-1.8) {$\{1,2,0,3\}$};
    \draw[white] (a1)--(c0) node[pos=0.5, inner sep=0.5pt,circle,fill=black] (a1c0) {\,};
    \node at (-1.4,-1.8) {$\{2,1,0,3\}$};
    \draw[white] (c0)--(b1) node[pos=0.5, inner sep=0.5pt,circle,fill=black] (b1c0) {\,};
    \node at (-2.7,-0.2) {$\{2,0,1,3\}$};
    \draw[white] (b1)--(a0) node[pos=0.5, inner sep=0.5pt,circle,fill=black] (a0b1) {\,};
    \node at (-1.4,2.0) {$\{0,2,1,3\}$};
    \draw[red, dashed,postaction={decorate}] (a0)--(a0c1) node[pos=0.5, inner sep=0.5pt, circle] {$\,$};
    \draw[red, dashed,postaction={decorate}] (a0)--(a0b1) node[pos=0.5, inner sep=0.5pt, circle] {$\,$};
    \draw[red, dashed,postaction={decorate}] (b0)--(b0c1) node[pos=0.5, inner sep=0.5pt, circle] {$\,$};
    \draw[red, dashed,postaction={decorate}] (b0)--(a1b0) node[pos=0.5, inner sep=0.5pt, circle] {$\,$};
    \draw[red, dashed,postaction={decorate}] (c0)--(a1c0) node[pos=0.5, inner sep=0.5pt, circle] {$\,$};
    \draw[red, dashed,postaction={decorate}] (c0)--(b1c0) node[pos=0.5, inner sep=0.5pt, circle] {$\,$};
    \draw[red, dashed,postaction={decorate}] (0)--(a1) node[pos=0.5, inner sep=0.5pt, circle] {$\,$};
    \draw[red, dashed,postaction={decorate}] (0)--(b1) node[pos=0.5, inner sep=0.5pt, circle] {$\,$};
    \draw[red, dashed,postaction={decorate}] (0)--(c1) node[pos=0.5, inner sep=0.5pt, circle] {$\,$};
    \draw[blue,postaction={decorate}] (c1)--(a0c1) node[pos=0.5, inner sep=0.5pt, circle] {$\,$};
    \draw[blue,postaction={decorate}] (c1)--(b0c1) node[pos=0.5, inner sep=0.5pt, circle] {$\,$};
    \draw[blue,postaction={decorate}] (a1)--(a1b0) node[pos=0.5, inner sep=0.5pt, circle] {$\,$};
    \draw[blue,postaction={decorate}] (a1)--(a1c0) node[pos=0.5, inner sep=0.5pt, circle] {$\,$};
    \draw[blue,postaction={decorate}] (b1)--(b1c0) node[pos=0.5, inner sep=0.5pt, circle] {$\,$};
    \draw[blue,postaction={decorate}] (b1)--(a0b1) node[pos=0.5, inner sep=0.5pt, circle] {$\,$};
    \draw[blue,postaction={decorate}] (0)--(a0) node[pos=0.5, inner sep=0.5pt, circle] {$\,$};
    \draw[blue,postaction={decorate}] (0)--(b0) node[pos=0.5, inner sep=0.5pt, circle] {$\,$};
    \draw[blue,postaction={decorate}] (0)--(c0) node[pos=0.5, inner sep=0.5pt, circle] {$\,$};
    \node[inner sep=0.5pt, circle,fill=black] (da0) at (0,6) {\,};
    \node at (0,6.3) {$\{0,123\}$};
    \draw[blue] (0,6) circle (0.15cm);
    \node[inner sep=0.5pt, circle,fill=black] (dc0) at (-5.196,-3) {\,};
    \node at (-5.5,-3.3) {$\{2,013\}$};
    \draw[blue] (-5.196,-3) circle (0.15cm);
    \node[inner sep=0.5pt, circle,fill=black] (db0) at (5.196,-3) {\,};
    \node at (5.5,-3.3) {$\{1,023\}$};
    \draw[blue] (5.196,-3) circle (0.15cm);
    \node[inner sep=0.5pt, circle,fill=black] (da1) at (0,-3) {\,};
    \node at (0,-3.3) {$\{12,03\}$};
    \draw[red, dashed] (0,-3) circle (0.15cm);
    \node[inner sep=0.5pt, circle,fill=black] (dc1) at (2.6,1.5) {\,};
    \node at (3.2,1.7) {$\{01,23\}$};
    \draw[red, dashed] (2.6,1.5) circle (0.15cm);
    \node[inner sep=0.5pt, circle,fill=black] (db1) at (-2.6,1.5) {\,};
    \node at (-3.2,1.7) {$\{20,13\}$};
    \draw[red, dashed] (-2.6,1.5) circle (0.15cm);
    \draw[white] (da0)--(dc1) node[pos=0.5, inner sep=0.5pt,circle,fill=black] (da0c1) {\,};
    \node at (2.0,4.0) {$\{0,1,23\}$};
    \draw[white] (dc1)--(db0) node[pos=0.5, inner sep=0.5pt,circle,fill=black] (db0c1) {\,};
    \node at (4.7,-0.7) {$\{1,0,23\}$};
    \draw[white] (db0)--(da1) node[pos=0.5, inner sep=0.5pt,circle,fill=black] (da1b0) {\,};
    \node at (2.8,-3.3) {$\{1,2,03\}$};
    \draw[white] (da1)--(dc0) node[pos=0.5, inner sep=0.5pt,circle,fill=black] (da1c0) {\,};
    \node at (-2.8,-3.3) {$\{2,1,03\}$};
    \draw[white] (dc0)--(db1) node[pos=0.5, inner sep=0.5pt,circle,fill=black] (db1c0) {\,};
    \node at (-4.7,-0.7) {$\{2,0,13\}$};
    \draw[white] (db1)--(da0) node[pos=0.5, inner sep=0.5pt,circle,fill=black] (da0b1) {\,};
    \node at (-2.0,4.0) {$\{0,2,13\}$};
    \draw[red, dashed,postaction={decorate}] (da0)--(da0c1) node[pos=0.5, inner sep=0.5pt, circle] {$\,$};
    \draw[red, dashed,postaction={decorate}] (da0)--(da0b1) node[pos=0.5, inner sep=0.5pt, circle] {$\,$};
    \draw[red, dashed,postaction={decorate}] (db0)--(db0c1) node[pos=0.5, inner sep=0.5pt, circle] {$\,$};
    \draw[red, dashed,postaction={decorate}] (db0)--(da1b0) node[pos=0.5, inner sep=0.5pt, circle] {$\,$};
    \draw[red, dashed,postaction={decorate}] (dc0)--(da1c0) node[pos=0.5, inner sep=0.5pt, circle] {$\,$};
    \draw[red, dashed,postaction={decorate}] (dc0)--(db1c0) node[pos=0.5, inner sep=0.5pt, circle] {$\,$};
    \draw[blue,postaction={decorate}] (dc1)--(da0c1) node[pos=0.5, inner sep=0.5pt, circle] {$\,$};
    \draw[blue,postaction={decorate}] (dc1)--(db0c1) node[pos=0.5, inner sep=0.5pt, circle] {$\,$};
    \draw[blue,postaction={decorate}] (da1)--(da1b0) node[pos=0.5, inner sep=0.5pt, circle] {$\,$};
    \draw[blue,postaction={decorate}] (da1)--(da1c0) node[pos=0.5, inner sep=0.5pt, circle] {$\,$};
    \draw[blue,postaction={decorate}] (db1)--(db1c0) node[pos=0.5, inner sep=0.5pt, circle] {$\,$};
    \draw[blue,postaction={decorate}] (db1)--(da0b1) node[pos=0.5, inner sep=0.5pt, circle] {$\,$};
    \draw[green!50!black,thick,dotted,postaction={decorate}] (da0)--(a0) node[pos=0.5, inner sep=0.5pt, circle] {$\,$};
    \draw[green!50!black,thick, dotted,postaction={decorate}] (db0)--(b0) node[pos=0.5, inner sep=0.5pt, circle] {$\,$};
    \draw[green!50!black,thick, dotted,postaction={decorate}] (dc0)--(c0) node[pos=0.5, inner sep=0.5pt, circle] {$\,$};
    \draw[green!50!black,thick, dotted,postaction={decorate}] (da1)--(a1) node[pos=0.5, inner sep=0.5pt, circle] {$\,$};
    \draw[green!50!black,thick, dotted,postaction={decorate}] (db1)--(b1) node[pos=0.5, inner sep=0.5pt, circle] {$\,$};
    \draw[green!50!black,thick, dotted,postaction={decorate}] (dc1)--(c1) node[pos=0.5, inner sep=0.5pt, circle] {$\,$};
    \draw[green!50!black,thick, dotted,postaction={decorate}] (da0c1)--(a0c1) node[pos=0.5, inner sep=0.5pt, circle] {$\,$};
    \draw[green!50!black,thick, dotted,postaction={decorate}] (da0b1)--(a0b1) node[pos=0.5, inner sep=0.5pt, circle] {$\,$};
    \draw[green!50!black,thick, dotted,postaction={decorate}] (db0c1)--(b0c1) node[pos=0.5, inner sep=0.5pt, circle] {$\,$};
    \draw[green!50!black,thick, dotted,postaction={decorate}] (da1c0)--(a1c0) node[pos=0.5, inner sep=0.5pt, circle] {$\,$};
    \draw[green!50!black,thick,dotted,postaction={decorate}] (db1c0)--(b1c0) node[pos=0.5, inner sep=0.5pt, circle] {$\,$};
    \draw[green!50!black,thick,dotted,postaction={decorate}] (da0c1)--(a0c1) node[pos=0.5, inner sep=0.5pt, circle] {$\,$};
    \draw[green!50!black,thick,dotted,postaction={decorate}] (db0c1)--(b0c1) node[pos=0.5, inner sep=0.5pt, circle] {$\,$};
    \draw[green!50!black,thick,dotted,postaction={decorate}] (da1b0)--(a1b0) node[pos=0.5, inner sep=0.5pt, circle] {$\,$};
    \draw[green!50!black,thick,dotted,postaction={decorate}] (da1c0)--(a1c0) node[pos=0.5, inner sep=0.5pt, circle] {$\,$};
    \draw[green!50!black,thick,dotted,postaction={decorate}] (db1c0)--(b1c0) node[pos=0.5, inner sep=0.5pt, circle] {$\,$};
    \end{scope}
\end{tikzpicture}}
\]
\end{ex}
We will show that the topological realisation of $\Sigma_k$ can be identified with a
$k$-simplex whose extreme points correspond to the $f \in \Sigma_k$ such that $\range(f)
= \{0, 1\}$ for $k \ge 1$ (of course $\range(f)=\{0\}$ if $k=0$, and then $\Sigma_0 = \Sigma_0^0$ is a point, as
is its topological realisation). We use the following notation throughout the construction.

\begin{ntn}
Let $\{\ep_0, \dots, \ep_k\}$ denote the usual basis vectors
for $\R^{k+1}$. Let $v_0 = \frac{1}{k+1} \sum^k_{i=0} \ep_i$, and for $f \in P_k$ and $n
\in \range(f) \setminus \{0\}$, define $v_{f,n} \in \R^{k+1}$ by
\[
v_{f,n} = \frac{1}{n} \sum_{f(j) < n} \ep_j.
\]
The defining property of placing functions ensures that $\|v_{f,n}\|_1 = 1$.

Recall that for $m \in \N^k$ we write $[0,m]$ for generalised interval $\{t \in \R^k : 0
\le t_i \le m_i\text{ for all }i \le k\}$. Given $f \in P_k$, define a map $\phi_f : [0,
h(f)] \to \R^{k+1}$ by $\phi_f(0) = v_0$ and
\begin{equation} \label{eq:phitform}
\phi_f(t) = (1 - \|t\|_\infty)v_0 + \frac{\|t\|_\infty}{\|t\|_1} \sum_{n \in \range(f)\setminus\{0\}} t_n v_{f,n}
\end{equation}
for $t \not= 0$.
\end{ntn}

\begin{thm}\label{thm:homeomorphism}
There is a homeomorphism $i : X_{\Sigma_k} \to \operatorname{conv}\{\ep_0, \ep_1, \dots,
\ep_k\}$ such that
\begin{equation}\label{eq:iformula}
i([(0,f), t]) = \phi_f(t)\quad\text{ for all $f \in P_k$ and $t \in [0,h(f)]$.}
\end{equation}
\end{thm}

The reader is invited to use \eqref{eq:phitform} to see that the
formula~\eqref{eq:iformula} carries the topological realisation of $\Sigma_2$ to the
simplex in $\R^3$ whose extreme points are the three unit basis vectors $\ep_0, \ep_1,
\ep_2$; for example, that the placing function $\{20,1\}$ is sent to $(\ep_0+\ep_2)/2$.

We need some technical lemmas to prove the theorem.

\begin{lem}\label{lem:technical 1}
Let $f \in P_k$ and $t \in [0, h(f)]$. For $i \le k$, we have $t_i \not= 0$ if and only
if there exists $r \in \R$ such that $|\{j : \phi_f(t)_j > r\}| = i$. Suppose that $|\{j
: \phi_f(t)_j > r\}| = i$, and define $g : \{0, \dots, k\} \to \{0, \dots, k\}$ by
\[
g(j) = \begin{cases}
        i&\text{ if $\phi_f(t)_j \le r$}\\
        0&\text{ if $\phi_f(t)_j > r$.}
    \end{cases}
\]
Then $g \in P_k$, $h(g) = \ep_i$ and $g \le f$.
\end{lem}

\begin{proof}
Choose a permutation $\sigma$ of $\{0, \dots, k\}$ such that $f \circ \sigma$ is
nondecreasing. Then each $(v_{f,n})_{\sigma(i)} = 1/n$ if $i \le n$ and $0$ if $i > n$,
and $(v_0)_{\sigma(i)} = 1/(k+1)$ for all $i$. In particular, each $w \in \{v_0\} \cup
\{v_{f,n} : n \le k\}$ satisfies $w_{\sigma(0)} \ge w_{\sigma(1)} \ge \dots \ge
w_{\sigma(k)}$, and $w_{\sigma(i)} = w_{\sigma(i+1)}$ unless $w = v_{f,i+1}$. Hence each
\begin{align*}
\phi_f(t)_{\sigma(i)}
    &= (1 - \|t\|_\infty)(v_0)_{\sigma(i)}
        + \frac{\|t\|_\infty}{\|t\|_1} \sum_{n \in \range(f)\setminus\{0\}} t_n (v_{f,n})_{\sigma(i)}\\
    &\ge (1 - \|t\|_\infty)(v_0)_{\sigma(i+1)}
        + \frac{\|t\|_\infty}{\|t\|_1} \sum_{n \in \range(f)\setminus\{0\}} t_n (v_{f,n})_{\sigma(i+1)},
\end{align*}
with strict inequality if and only if $t_{i+1} > 0$ (since $t \le h(f)$, this forces $i +
1 \in \range(f)\setminus\{0\}$). This proves the first assertion.

To prove the second assertion, observe that $g \in P_k$ by choice of $r$, and $h(g) =
e_i$ by definition of $h$. We claim that whenever $\phi_f(t)_{j'} > \phi_f(t)_j$, we have
$f(j') < f(j)$. To prove this, first suppose that $j,j' \le k$ satisfy $\phi_f(t)_{\sigma(j')}
> \phi_f(t)_{\sigma(j)}$. Then the preceding paragraph shows that $j' < j$ and $t_{l+1} >
0$ for some $j' \le l < j$. Since $f \circ \sigma$ is a nondecreasing placing function,
it is dominated in the ordering on $P_k$ by the identity permutation. So
Lemma~\ref{lem:tailfactor} implies that
\[
    f \circ \sigma(j') = \max\{n \le j' : h(f)_n = 1\} \le j' \le l.
\]
Since $t < h(f)$, we also have
\[
f \circ \sigma(j)
    = \max\{n \le j : h(f)_n = 1\}
    \ge \max\{n \le j : t_n > 0\}
    \ge l+1.
\]
In particular, $f\circ\sigma(j') < f\circ\sigma(j)$. Since $\sigma$ is invertible, we
deduce that $\phi_f(t)_{j'} > \phi_f(t)_j$ implies $f(j') < f(j)$ for all $j,j'$. This
proves the claim.

To show that $g \le f$, observe that $g(j) = 0 \le f(j)$ whenever $\phi_f(t)_j > r$;
and if $\phi_f(t)_j \le r$, then the claim gives
\[
f(j)
    = |\{l : f(l) < f(j)\}|
    \ge |\{l : \phi_f(t)_l > \phi_f(t)_j\}|
    \ge |\{l : \phi_f(t)_l > r\}|
    = i = g(j).\qedhere
\]
\end{proof}

For the next lemma, if $\Lambda$ is a $k$-graph and $F$ is a finite subset of $\Lambda$,
then we define
\[\textstyle
\MCE(F) = \big\{\lambda \in \Lambda : \lambda \in \alpha\Lambda\text{ for all }
    \alpha, \in F\text{ and }d(\lambda) = \bigvee_{\alpha \in F} d(\alpha)\big\} ,
\]
so we have $\MCE(\alpha, \beta) = \MCE(\{\alpha,\beta\})$ for all $\alpha,\beta \in \Lambda$.
If $F$ is a finite subset of $\Lambda$, $\lambda \in F$ and $G = F \setminus\{\lambda\}$,
then
\begin{equation}\label{eq:MCE inducts}
\MCE(F) = \bigcup_{\mu \in \MCE(G)} \MCE(\lambda, \mu).
\end{equation}

\begin{lem}\label{lem:MCEs}
For $F \subseteq \{(0,f) : f \in P_k\} \subseteq \Sigma_k$, we have $|\MCE(F)| \le 1$.
\end{lem}

\begin{proof}
We proceed by induction on $|F|$; if $|F| \le 1$ then the statement is trivial.

Suppose that $|F| = 2$, say $F = \{(0,f), (0,g)\}$. We have $\MCE((0,f), (0,g)) = \{(0,
a) : h(a) = h(f) \vee h(g)\text{ and } f,g \le a\}$. So suppose that $a \in P_k$
satisfies $h(a) = h(f) \vee h(g)$ and $f,g \le a$. Choose $\sigma \in P^{\max}_k$ such
that $a \le \sigma$. Since $\sigma$ is surjective, Lemma~\ref{lem:tailfactor} implies
that for $j \le k$,
\begin{gather*}
f(j) = \max\{j \le \sigma(i) : h(f)_j = 1\},\quad
    g(j) = \max\{j \le \sigma(i) : h(g)_j = 1\},\quad\text{ and }\\
    a(j) = \max\{j \le \sigma(i) : h(f)_j = 1\text{ or } h(g)_j = 1\}.
\end{gather*}
So $a(j) = \max\{f(j), g(j)\}$ for all $j$, and in particular, $a$ is uniquely determined by $f,g$.

Now suppose that $|\MCE(G)| \le 1$ whenever $|G| < |F|$. Fix $\lambda \in F$ and let $G =
F \setminus\{\lambda\}$. Property~\eqref{eq:MCE inducts} above gives $\MCE(F) =
\bigcup_{\mu \in \MCE(G)} \MCE(\lambda, \mu)$. The inductive hypothesis implies that
there is at most one $\mu \in \MCE(G)$, and then the result follows from the base case.
\end{proof}

\begin{lem}\label{lem:phif-inj}
For $f \in P_k$ the map $\phi_f$ of Equation~\ref{eq:phitform} is injective.
\end{lem}

\begin{proof}
We have $\phi_f(s) = v_0$ if and only if $s = 0$, so we suppose that $s, t \not= 0$ and
$\phi_f(s) = \phi_f(t)$ and show that $s = t$. The vectors $v_0, v_{f, 1}, \dots,
v_{f,n}$ are linearly independent, so $\phi_f(s) = \phi_f(t)$ implies $1 - \|s\|_\infty =
1 - \|t\|_\infty$ and $\frac{\|s\|_\infty}{\|s\|_1} s_n = \frac{\|t\|_\infty}{\|t\|_1}
t_n$ for all $n$. In particular, $s = \lambda t$ for some $\lambda \in (0, \infty)$, and
$\lambda \|t\|_\infty = \|\lambda t\|_\infty = \|s\|_\infty = \|t\|_\infty$. As $s
\not= 0$ we have $\lambda = 1$ and hence $s = t$.
\end{proof}

\begin{proof}[Proof of Theorem~\ref{thm:homeomorphism}]
Each $\phi_f$ is a continuous map. By definition of $X_{\Sigma_k}$, each point in
$X_{\Sigma_k}$ has the form $[(g,f), t]$ where $g \le f \in P_k$ and $h(g) \le t \le
h(f)$. Since $((g,f), t) \sim ((0, f), t + h(g))$, it follows that every point in
$X_{\Sigma_k}$ can be expressed as $[(0,f), s]$ for some $f \in P_k$ and $s \in [0,
h(f)]$.

So it suffices to show that $[(0,f), s] = [(0,g), t]$ if and only if $\phi_f(s) =
\phi_g(t)$; for then~\eqref{eq:iformula} descends to a continuous bijection between
$X_{\Sigma_k}$ and $\operatorname{conv}\{\ep_0, \ep_1, \dots, \ep_k\}$, which is then a
homeomorphism because both sets are compact and Hausdorff.

First suppose that $[(0,f), s] = [(0, g), t]$. Then $s - \floor{s} = t - \floor{t}$, and
$(0,f)(\floor{s}, \ceil{s}) = (0,g)(\floor{t}, \ceil{t})$. Hence
\[
\floor{s}
    = h\big(r\big((0,f)(\floor{s}, \ceil{s})\big)\big)
    = h\big(r\big((0,g)(\floor{t}, \ceil{t})\big)\big)
    = \floor{t},
\]
and combining this with $s - \floor{s} = t - \floor{t}$ gives $s = t$. By the
factorisation property there is a unique factorisation $(0,f)=\alpha\beta$ in $\Sigma_k$
with $d(\alpha)=\ceil{s}$. Define $f \wedge g$ to be the placing function for
$s(\alpha)$. Then $f \wedge g \le f,g$, and
\[
[(0,f), s] = [(0, f \wedge g), s] = [(0,g), t].
\]
So we may assume without loss of generality that $g \le f$, and so $s \le h(g) \le h(f)$.
Whenever $s_n \not= 0$ we have $h(g)_n = h(f)_n = 1$ and hence $n \in \range(g) \subseteq
\range(f)$. So the final assertion of Lemma~\ref{lem:tailfactor} implies that
$g^{-1}(\{0, \dots, n-1\}) = f^{-1}(\{0, \dots, n-1\})$, and hence $v_{g,n} = v_{f,n}$.
So $\phi_g(s) = \phi_f(s)$ as required.

Now suppose that $\phi_f(s) = \phi_g(t)$. The first assertion of Lemma~\ref{lem:technical
1} implies that $s_i > 0$ if and only if $t_i > 0$, and so $\ceil{s} = \ceil{t}$. The
second assertion of Lemma~\ref{lem:technical 1} combined with the factorisation property
in $\Sigma_k$ implies that for each $i$ such that $t_i > 0$ we have $d(0,f), d(0,g) \ge
e_i$ and the unique edge $\alpha_i \in \Sigma_k^{e_i}$ such that $(0,f) = \alpha_i
\tau_i$ for some $\tau_i$ also satisfies $(0,g) = \alpha_i \rho_i$ for some $\rho_i$. By
Lemma~\ref{lem:tailfactor} there are unique elements $f', g' \in P_k$ such that $h(f') =
h(g') = \ceil{s}$ and $f' \le f$ and $g' \le g$. The factorisation property implies that
$(0,f'), (0,g') \in \alpha_i \Sigma_k$ for all $i$. Since $d((0,f')) = d((0,g')) =
\bigvee_{t_i > 0} d(\alpha_i)$, we then have $(0,f'), (0,g') \in \MCE(\{\alpha_i : t_i >
0\})$. So Lemma~\ref{lem:MCEs} implies that $|\MCE(\{\alpha_i : t_i
> 0\})| \le 1$, and we deduce that $f' = g'$. The argument of the preceding paragraph
shows that
\[
\phi_{f'}(s) = \phi_f(s) = \phi_g(t) = \phi_{g'}(t),
\]
and since $f' = g'$, Lemma~\ref{lem:phif-inj} implies that $s = t$. Hence
\[
((0, f),s) \sim ((0, f'), s) = ((0, g'), t) \sim ((0, g), t). \qedhere
\]
\end{proof}

\begin{proof}[Proof of Theorem~\ref{thm:spheres realised}]
Consider the set $\{0, 1\}$, regarded as a $0$-graph, and the $k$-graph $\Sigma_k$
described in Theorem~\ref{thm:homeomorphism}. Form the cartesian product $k$-graph
$\Lambda = \{0, 1\} \times \Sigma_k$. Define a relation $\sim$ on $\Lambda$ by $(i,
(f,g)) \sim (j, (f',g'))$ if and only if $(f,g) = (f', g') \in \Sigma_k$ and $f \not= 0$.
It is straightforward to check that this relation satisfies conditions
(\ref{it:sim-d})--(\ref{it:sim-lift}) of Proposition~\ref{prp:quotient k-graph}, and so
we may form the quotient $k$-graph $\Gamma = \Lambda/{\sim}$.

We show that $X_\Gamma$ is homeomorphic to a $k$-sphere. Proposition~\ref{prp:topological
gluing} shows that $X_\Gamma$ is the quotient of $\{0,1\} \times
X_{\Sigma_k}$ by the equivalence relation $(i, [(f,g), t]) \sim (j, [(f,g), t])$ whenever
$f \not= 0$.

We claim that the homeomorphism $i : X_{\Sigma_k} \to \conv\{\ep_0, \dots, \ep_k\}
\subseteq \R^{k+1}$ carries $\{[(f,g), t] : f \not= 0\}$ to $\bigcup^k_{i=1}
\conv\{\ep_0, \dots, \ep_{i-1}, \ep_{i+1}, \dots, \ep_k\}$, which is the surface of the
$k$-simplex with extreme points $\ep_0, \dots, \ep_k$. To see this, observe that
\eqref{eq:phitform} implies that for any $f \in P_k$ and $n \in \range(f) \setminus
\{0\}$, we have $v_{f,n} \in \conv\{\ep_j : f(j) < \|f\|_{\infty}\}$. There exists some
$i_f$ such that $f(i_f) = \|f\|_\infty$, and so
\[
\conv\{v_{f,n} : n \in \range{f} \setminus \{0\}\}
     \subseteq \conv\{\ep_0, \dots, \ep_{i-1}, \ep_{i+1}, \dots, \ep_k\}.
\]
Consider $i([(f,g), t])$ where $f \not= 0$. Since some $f(i) > 0$ we have $h(f) > 0$ and
therefore $t_j = 1$ for some $j$. In particular $\|t\|_\infty = 1$ and
so~\eqref{eq:phitform} implies that $i([(f,g), t])$ belongs to $\conv\{v_{f,n} : n \in
\range{f} \setminus \{0\}\}$. This gives $\{[(f,g), t] : f \not= 0\} \subseteq
\bigcup^k_{i=1} \conv\{\ep_0, \dots, \ep_{i-1}, \ep_{i+1}, \dots, \ep_k\}$. For the
reverse inclusion, fix $\sigma \in P_k^{\max}$ and $f \in P_k$ with $f \le \sigma$ and
$h(f) = \ep_i$. Let $j = \sigma^{-1}(k)$. Using what we have already proved, we see that
\begin{align*}
i(\{[(f,\sigma), t] : {}&t \in [e_i, \mathbf{1}_k]\})\\
    &= \{t \in \conv\{\ep_1, \dots, \ep_{j-1}, \ep_{j+1}, \dots, \ep_k\} :
            t_{\sigma^{-1}(l)} \ge t_{\sigma^{-1}(l+1)}\text{ for all } l\}.
\end{align*}
For $\sigma \in P^{\max}_k$ and $j \le k$, let $f_{\sigma, j}$ be the unique element of
$P_k$ with $f_{\sigma, j} \le \sigma$ and $h(f_{\sigma, j}) = e_j$. Then
\begin{align*}
i(\{[(f,g), t] : f \not= 0\})
    &= \bigcup_{\sigma \in P^{\max}_k} \bigcup_{j \le k}
        i(\{[(f_{\sigma, j},\sigma), t] : t \in [e_j, \mathbf{1}_k]\}) \\
   &= \bigcup^k_{i=1} \conv\{\ep_0, \dots, \ep_{i-1}, \ep_{i+1}, \dots, \ep_k\},
\end{align*}
which proves the claim.

So $X_\Gamma$ is the topological disjoint union of two rank-$k$
simplices (which are homeomorphic to $k$-spheres) glued along their common
$(k-1)$-dimensional boundary; that is, a $k$-sphere.
\end{proof}

\begin{cor}
Fix $n \ge 1$ and $k \ge 0$. There is a finite $k$-graph whose topological realisation is
homeomorphic to a wedge of $n$ $k$-spheres.
\end{cor}

\begin{proof}
Let $\Gamma$ be the $k$-graph of Theorem~\ref{thm:spheres realised}. Regard $\{1 , \ldots
, n \}$ as a $0$-graph so that $\{0, \dots, n\} \times \Gamma$ is a $k$-graph. Each
$\{i\} \times \Gamma$ is a quotient of $\{0,1\} \times \Sigma_k$. The vertex $(0,0)$ in
$\Sigma_k$ satisfies $(0,0)\Sigma_k = \{(0,0)\}$, and so the vertex $v := [0, (0,0)]$ of
$\Gamma$ has the same property. For each $i \le n$, let $v_i := (i, v)$ be the copy of
$v$ in $\{i\} \times \Gamma$. Define an equivalence relation on $\{0, \dots, n\} \times
\Gamma$ by $\alpha \sim \beta$ if and only if $\alpha = v_i$ and $\beta = v_j$ for some
$i,j$. Again, this relation satisfies (\ref{it:sim-d})--(\ref{it:sim-lift}) of
Proposition~\ref{prp:quotient k-graph}, and so we may form the quotient $k$-graph
$\big(\{0, \dots, n\} \times \Gamma\big)/{\sim}$. Theorem~\ref{prp:topological gluing}
implies that the topological realisation of this quotient is the disjoint union of the
$X_{\{i\} \times \Gamma}$ all glued at a single point. We saw in the proof of
Theorem~\ref{thm:spheres realised} that each $X_{\{i\} \times \Gamma}$ is a $k$-sphere,
so $X_{\big(\{0, \dots, n\} \times \Gamma\big)/{\sim}}$ is a homeomorphic to a wedge of
$n$ $k$-spheres.
\end{proof}

\end{document}